\documentclass[11pt]{amsart}

\usepackage{amsfonts, amstext, amsmath, amsthm, amscd, amssymb}
\usepackage{epsfig, graphics, psfrag, overpic}
\usepackage{color}
\usepackage[
pdfborderstyle={},
pdfborder={0 0 0},
pagebackref,
pdftex]{hyperref}

\renewcommand*{\backref}[1]{}
\renewcommand*{\backrefalt}[4]{
  \ifcase #1 %
   [No citations.]%
  \or
   [#2]%
  \else
   [#2]%
  \fi
}

\let\oldmarginpar\marginpar
\renewcommand\marginpar[1]{\oldmarginpar[\raggedleft\footnotesize #1]%
{\raggedright\footnotesize #1}}

\renewcommand{\setminus}{{\smallsetminus}}

\newcommand{\HH}{{\mathbb{H}}}
\newcommand{\RR}{{\mathbb{R}}}
\newcommand{\ZZ}{{\mathbb{Z}}}
\newcommand{\NN}{{\mathbb{N}}}

\newcommand{\st}{\mathbin{\mid}} 

\newcommand{\cross}{\mathbin{\times}}
\newcommand{\vol}{{\operatorname{vol}}}
\newcommand{\area}{{\operatorname{area}}}
\newcommand{\height}{{\operatorname{height}}}
\newcommand{\width}{{\operatorname{width}}}
\newcommand{\bdy}{{\partial}}
\newcommand{\homeo}{\mathrel{\cong}} 
\newcommand{\isom}{\mathrel{\cong}} 
\newcommand{\from}{\colon} 
\newcommand{\abs}[1]{{\left\vert #1 \right\vert}}
\newcommand{\mon}{{\varphi}}

\newcommand{\quotient}[2]{{\raisebox{0.2em}{$#1$}
           \!\!\left/\raisebox{-0.2em}{\!$#2$}\right.}}

\newcommand{\curve}{{\mathcal{C}}}
\newcommand{\arc}{{\mathcal{A}}}
\newcommand{\pants}{{\mathcal{P}}}
\newcommand{\teich}{{\mathcal{T}}}
\newcommand{\dbar}{{\overline{d}}}
\newcommand{\core}{{\rm core}}
\newcommand{\eps}{{\varepsilon}}
\newcommand{\calS}{{\mathcal{S}}}
\newcommand{\MCG}{{\operatorname{MCG}}}
\newcommand{\GL}{\mathcal{GL}} 
\newcommand{\ML}{\mathcal{ML}} 
\newcommand{\PML}{\mathcal{PML}} 
\newcommand{\lamm}{{L}} 

\newcommand{\Image}{{\operatorname{Im}}}

\theoremstyle{plain}
\newtheorem{theorem}{Theorem}[section]
\newtheorem{corollary}[theorem]{Corollary}
\newtheorem{lemma}[theorem]{Lemma}
\newtheorem{prop}[theorem]{Proposition}

\newtheorem*{namedtheorem}{\theoremname}
\newcommand{\theoremname}{testing}
\newenvironment{named}[1]{\renewcommand{\theoremname}{#1}\begin{namedtheorem}}{\end{namedtheorem}}

\theoremstyle{definition}
\newtheorem{define}[theorem]{Definition}
\newtheorem{remark}[theorem]{Remark}
\newtheorem{fact}[theorem]{Fact}

\newtheorem*{proofclaim}{Claim}

\newtheorem{question}[theorem]{Question}

\numberwithin{equation}{section}

\begin{document}
\title{Cusp geometry of fibered $3$--manifolds}
\author{David Futer}
\address[]{Department of Mathematics, Temple University,
Philadelphia PA 19122, USA}
\email[]{dfuter@temple.edu}

\author{Saul Schleimer}
\address[]{Department of Mathematics, 
University of Warwick, Coventry CV4 7AL, UK}
\email[]{s.schleimer@warwick.ac.uk}

\thanks{{Futer is supported in part by NSF grant DMS--1007221.}}
\thanks{{Schleimer is supported in part by EPSRC grant EP/I028870/1.}}
\thanks{{This work is in the public domain.}}

\date{ \today}

\begin{abstract}  
Let $F$ be a surface and suppose that $\mon \from F \to F$ is a
pseudo-Anosov homeomorphism, fixing a puncture $p$ of $F$.  The
mapping torus $M = M_\mon$ is hyperbolic and contains a maximal cusp
$C$ about the puncture $p$.

We show that the area (and height) of the cusp torus $\bdy C$ is equal
to the stable translation distance of $\mon$ acting on the arc complex
$\arc(F,p)$, up to an explicitly bounded multiplicative error.  Our
proof relies on elementary facts about the hyperbolic geometry of
pleated surfaces.  In particular, the proof of this theorem does not
use any deep results from Teichm\"uller theory, Kleinian group theory,
or the coarse geometry of $\arc(F,p)$.

A similar result holds for quasi-Fuchsian manifolds $N \homeo F \times
\RR$.  In that setting, we find a combinatorial estimate for the area
(and height) of the cusp annulus in the convex core of $N$, up to
explicitly bounded multiplicative and additive error.  As an
application, we show that covers of punctured surfaces induce
quasi-isometric embeddings of arc complexes.
\end{abstract}

\maketitle

\section{Introduction}

Following the work of Thurston, Mostow, and Prasad, it has been known
for over three decades that almost every $3$--manifold with torus
boundary admits a hyperbolic structure~\cite{thurston:survey}, which
is necessarily unique up to isometry~\cite{mostow:rigidity,
  prasad:rigidity}.  Thus, in principle, it is possible to translate
combinatorial data about a $3$--manifold into a detailed description
of its geometry --- and conversely, to use geometry to identify
topological features. Indeed, given a triangulated manifold (up to
over 100 tetrahedra) the computer program SnapPy can typically
approximate the manifold's hyperbolic metric to a high degree of
precision~\cite{snappea}.  However, building an effective dictionary
between combinatorial and geometric features, for all but the most
special families of manifolds, has proven elusive.  The prevalence of
hyperbolic geometry makes this one of the central open problems in
low-dimensional topology.

\subsection{Fibered $3$--manifolds}
In this paper, we attack this problem for the class of hyperbolic
$3$--manifolds that fiber over the circle.  Let $F$ be a connected,
orientable surface with $\chi(F) < 0$, and (for this paper) with at
least one puncture.  Given an orientation-preserving homeomorphism
$\mon \from F \to F$ we construct the {\em mapping torus}
\[
M_\mon := \,\quotient{F \times [0,1]}{(x, 1) \sim (\mon(x), 0)}.
\]
Thus $M_\mon$ fibers over $S^1$, with fiber $F$ and monodromy $\mon$.
Thurston showed that $M_\mon$ is hyperbolic if and only if $\mon$ is
\emph{pseudo-Anosov}: equivalently, if and only if $\mon^n(\gamma)$ is
not homotopic to $\gamma$ for any $n \neq 0$ and any essential simple
closed curve $\gamma \subset F$
\cite{thurston:survey,thurston:geometry-dynamics}. See also Otal
\cite{otal:fibered}.  In addition to the connection with dynamics,
fibered $3$--manifolds are of central importance in low-dimensional
topology because every finite-volume, non-positively curved
$3$--manifold has a finite--sheeted cover that fibers
\cite{agol:virtual-haken, przytycki-wise:mixed-manifolds}.

For those fibered $3$--manifolds that are hyperbolic, the work of
Minsky, Brock, and Canary on Kleinian surface groups provides a
combinatorial, bi-Lipschitz model of the hyperbolic metric
\cite{minsky:models-bounds}. The bi-Lipschitz constants depend only on
the fiber $F$~\cite{brock-canary-minsky:elc}. However, the existence
of these constants is proved using compactness arguments; as a result
the constants are unknown.

Using related ideas, Brock established the following notable entry in
the dictionary between combinatorics and geometry.

\begin{theorem}[Brock~\cite{brock:fibered, brock:quasifuchsian}]
\label{thm:pants-volume}
Let $F$ be an orientable surface with $\chi(F)<0$.  Then there exist
positive constants $K_1$ and $K_2$, depending only on $F$, such that
the following holds.  For every orientation-preserving, pseudo-Anosov
homeomorphism $\mon \from F \to F$, the mapping torus $M_\mon$ is a
hyperbolic $3$--manifold satisfying
\[
K_1 \, \dbar_\pants(\mon) \: \leq \: \vol(M_\mon) \: 
     \leq \: K_2 \, \dbar_\pants(\mon).
\]
\end{theorem}

\noindent
Here $\pants(F)$ is the adjacency graph of pants decompositions of
$F$.  Also $\dbar_\pants(\mon)$ is the stable translation distance of
$\mon$ in $\pants(F)$, defined in Equation~\eqref{eq:stable-distance}
below.

The constant $K_2$ in the upper bound can be made explicit.  Agol
showed that the sharpest possible value for $K_2$ is $2v_8$, where
$v_8 = 3.6638...$ is the volume of a regular ideal octahedron
\cite{agol:small}.  On the other hand, the constant $K_1$ is only
known in the special case when $F$ is a punctured torus or
$4$--puncture sphere; see Gu\'eritaud and Futer~\cite[Appendix
  B]{gf:punctured-torus}. For all other surfaces, it remains an open
problem to give an explicit estimate for $K_1$.

Brock's theorem is a template for obtaining combinatorial information;
the pants graph $\pants(F)$ is just one of many complexes naturally
associated to a surface $F$.  Others include the curve complex
$\curve(F)$ and the arc complex $\arc(F)$; the latter is the main
focus of this paper.  Using $\arc(F)$ we give \emph{effective}
two-sided estimates for the geometry of maximal cusps in $M_\mon$.

\begin{define}
\label{def:arc-complex}
Suppose $F$ is a surface of negative Euler characteristic, connected
and orientable, without boundary and with at least one puncture.  The
\emph{arc complex} $\arc(F)$ is the simplicial complex whose vertices
are proper isotopy classes of essential arcs from puncture to
puncture.  Simplices are collections of vertices admitting pairwise
disjoint representatives.  We engage in the standard abuse of notation
by using the same symbol for an arc and its isotopy class.

The $1$--skeleton $\arc^{(1)}(F)$ has a combinatorial metric.  For a
pair of vertices $v, w \in \arc^{(0)}(F)$, the distance $d(v,w)$ is
the minimal number of edges required to connect $v$ to $w$.  This is
well-defined, because $\arc(F)$ is
connected~\cite{hatcher:triangulations}.



When $F$ has a preferred puncture $p$,  we define the subcomplex
$\arc(F,p) \subset \arc(F)$ whose vertices are arcs with at least one
endpoint at $p$.  The $1$--skeleton $\arc^{(1)}(F,p)$ is again
connected.  The distance $d_\arc(v,w)$ is the minimal number of edges
required to connect $v$ to $w$ inside of $\arc^{(1)}(F,p)$.
\end{define}

The mapping class group $\MCG(F)$ acts on $\arc(F)$ by isometries. In
fact, Irmak and McCarthy showed~\cite{irmak-mccarthy:arc-complex}
that, apart from a few low-complexity exceptions, $\MCG(F) \isom
\operatorname{Isom} \arc(F)$.  Similarly, the subgroup of $\MCG(F)$
that fixes the puncture $p$ acts on $\arc(F,p)$ by isometries.  We are
interested in the geometric implications of this action.

\begin{define}
\label{def:trans-distance}
Let $\mon \from F \to F$ be a homeomorphism fixing $p$. Define the
\emph{translation distance} of $\mon$ in $\arc(F,p)$ to be
\begin{equation}\label{eq:trans-distance}
d_\arc(\mon) = \min \{ d_\arc(v, \mon(v)) \st v \in \arc^{(0)}(F,p) \}.
\end{equation}
The same definition applies in any simplicial complex where $\MCG(F)$
acts by isometries.

We also define the \emph{stable translation distance} of $\mon$ to be
\begin{equation}\label{eq:stable-distance}
\dbar_\arc(\mon) = \lim_{n \to \infty} \frac{ d_\arc(v, \mon^n(v))
}{n}, \quad \mbox{for an arbitrary vertex } v \in \arc^{(0)}(F,p),
\end{equation} 
and similarly for other $\MCG$--complexes. It is a general property of
isometries of metric spaces that the limit in
\eqref{eq:stable-distance} exists and does not depend on the base
vertex $v$~\cite[Section 6.6]{bridson-haefliger}. In addition, the
triangle inequality implies that $\dbar_\arc(\mon) \leq d_\arc(\mon)$.
\end{define}

Note that applying equation \eqref{eq:stable-distance} to the pants
graph $\pants(F)$ gives the stable translation distance $\dbar_\pants(\mon)$
that estimates volume in Theorem~\ref{thm:pants-volume}.
In the same spirit, one may ask the following.

\begin{question}\label{quest:dynamics}
Let $\calS(F)$ be a simplicial complex associated to a surface $F$, on
which the mapping class group $\MCG(F)$ acts by isometries.  How are the
dynamics of the action of $\mon \in \MCG(F)$ on $\calS(F)$ reflected
in the geometry of the mapping torus $M_\mon$?
\end{question}

We answer this question for the arc complex of a once-punctured
surface $F$ or, more generally, for the sub-complex $\arc(F,p)$ of a
surface with many punctures. Here, the stable distance
$\dbar_\arc(\mon)$ predicts the cusp geometry of $M_\mon$.

\subsection{Cusp area from the arc complex}

Let $M$ be a $3$--manifold whose boundary is a non-empty union of
tori, such that the interior of $M$ supports a complete hyperbolic
metric. In this metric, every non-compact end of $M$ is a \emph{cusp},
homeomorphic to $T^2 \times [0, \infty)$. Geometrically, each cusp is
a quotient of a horoball in $\HH^3$ by a $\ZZ \times \ZZ$ group of
deck transformations. We call this geometrically standard end a
\emph{horospherical cusp neighborhood} or \emph{horocusp}.

Associated to each torus $T \subset \bdy M$ is a \emph{maximal cusp}
$C = C_T$. That is, $C$ is the closure of $C^\circ \subset M$, where
$C^\circ$ is the largest embedded open horocusp about $T$. 
The same construction works in dimension $2$: every punctured
hyperbolic surface has a maximal cusp about each puncture.

In dimension $3$, Mostow--Prasad rigidity implies that the geometry of
a maximal cusp $C \subset M$ is completely determined by the topology
of $M$. One may compute that $\area(\bdy C) = \frac{1}{2}
\vol(C)$. The Euclidean geometry of $\bdy C$ is an important invariant
that carries a wealth of information about Dehn fillings of $M$. For
example, if a \emph{slope} $s$ (an isotopy class of simple closed
curve on $\bdy C$) is sufficiently long, then Dehn filling $M$ along
$s$ produces a hyperbolic manifold~\cite{agol:6theorem,
lackenby:surgery}, whose volume can be estimated in terms of the
length $\ell(s)$~\cite{fkp:volume}.

In our setting, $M_\mon$ is a fibered hyperbolic $3$--manifold, with
fiber a punctured surface $F$. The maximal cusp torus $\bdy C$
contains a canonical slope, called the \emph{longitude} of $C$, which
encircles a puncture of $F$.  The Euclidean length of the longitude is
denoted $\lambda$.  Any other, non-longitude slope on $\bdy C$ must
have length at least
\begin{equation}
\label{eq:height}
\height(\bdy C) := \area(\bdy C) / \lambda.
\end{equation}
As discussed in the previous paragraph, lower bounds on $\height(\bdy
C)$ imply geometric control over Dehn fillings of $M_\mon$.
 
Our main result in this paper uses the action of $\mon$ on $\arc(F,p)$
to give explicit estimates on the area and height of the cusp torus
$\bdy C$.

\begin{theorem}
\label{thm:main}
Let $F$ be a surface with a preferred puncture $p$, and let $\mon
\from F \to F$ be any orientation-preserving, pseudo-Anosov
homeomorphism. In the mapping torus $M_\mon$, let $C$ be the maximal
 cusp that corresponds to $p$. Let
$\psi = \mon^n$ be the smallest positive power of $\mon$ with the
property that $\psi(p)=p$. Then

\[
\frac{ \dbar_{\arc} (\psi)}{450\, \chi(F)^4 } \: < \:
\area(\bdy C) \: \leq \: 9 \, \chi(F)^2 \, \dbar_{\arc} (\psi).
\]
Similarly, the height of the cusp relative to a longitude satisfies
\[
\frac{ \dbar_{\arc} (\psi)}{536\, \chi(F)^4 } \: < \: 
\height(\bdy C) \: < \: -3 \, \chi(F) \, \dbar_{\arc} (\psi).
\]
\end{theorem}

If the surface $F$ has only one puncture $p$, the statement of Theorem
\ref{thm:main} becomes simpler in several ways. In this special case,
we have $n = 1$, hence $\psi = \mon$. There is only one cusp in
$M_\mon$, and  $\arc(F,p) = \arc(F)$. In this special case, the area
and height of the maximal cusp are estimated by the stable translation
distance $\dbar_\arc(\mon)$, acting on $\arc(F)$.

In the special case where $F$ is a once-punctured torus or
$4$--punctured sphere, Futer, Kalfagianni, and Purcell proved a
similar estimate, with sharper constants. See~\cite[Theorems 4.1 and
4.7]{fkp:farey}. Theorem~\ref{thm:main} generalizes those
results to \emph{all} punctured hyperbolic surfaces.

We note that a non-effective version of Theorem~\ref{thm:main} can be
derived from Minsky's \emph{a priori bounds} theorem for the length of
curves appearing in a hierarchy~\cite[Lemma
7.9]{minsky:models-bounds}. In fact, this line of argument was our
original approach to estimating cusp area. In the process of studying
this problem, we came to realize that arguments using the geometry and
hierarchical structure of the curve complex $\curve(F)$ can be
replaced by elementary arguments focusing on the geometry of pleated
surfaces. See Section~\ref{subsec:outline} below for an outline of
this effective argument.

\subsection{Quasi-Fuchsian $3$--manifolds}

The methods used to prove Theorem~\ref{thm:main} also apply to
quasi-Fuchsian manifolds. We recall the core definitions; see
Marden~\cite[Chapter 3]{marden:book} or
Thurston~\cite[Chapter~8]{thurston:notes} for more details.  A
hyperbolic manifold $N = \HH^3 / \Gamma$ is called
\emph{quasi-Fuchsian} if the limit set $\Lambda(\Gamma)$ of $\Gamma$
is a Jordan curve on $\bdy \HH^3$, and each component of $\bdy \HH^3
\setminus \Lambda(\Gamma)$ is invariant under $\Gamma$.  In this case,
$N$ is homeomorphic to $F \times \RR$ for a surface $F$. The
\emph{convex core} of $N$, denoted $\core(N)$, is defined to be the
quotient, by $\Gamma$, of the convex hull of the limit set
$\Lambda(\Gamma)$.  
When $N$ is quasi-Fuchsian but not Fuchsian, $\core(N) \cong F \times
[0,1]$, and its boundary is the disjoint union of two surfaces
$\bdy_+\core(N)$ and $\bdy_-\core(N)$, each intrinsically hyperbolic,
and each pleated along a lamination. See Definition~\ref{def:pleated}.

Although the quasi-Fuchsian manifold $N$ has infinite volume, the
volume of $\core(N)$ is finite. Each puncture of $F$ corresponds to a
rank one maximal cusp $C$ (the quotient of a horoball by
$\mathbb{Z}$), such that $C \cap \core(N)$ has finite volume and $\bdy
C \cap \core(N) \cong S^1 \times [0,1]$ has finite area.  Thus we may
attempt to estimate the area and height of $\bdy C \cap \core(N)$
combinatorially.

\begin{define}
\label{def:distance-qf}
Let $N \cong F \times \RR$ be a quasi-Fuchsian $3$--manifold, and let
$p$ be a puncture of $F$.  Define $\Delta_+(N)$ to be the collection of
all shortest arcs from $p$ to $p$ in $\bdy_+ \core(N)$.  By
Lemma~\ref{lemma:shortish-disjoint}, the arcs in $\Delta_+(N)$ are
pairwise disjoint, so $\Delta_+(N)$ is a simplex in $\arc(F, p)$.
Similarly, let $\Delta_-(N)$ be the simplex of shortest arcs from $p$
to $p$ in $\bdy_- \core(N)$.

We define the \emph{arc distance} of $N$ relative to the puncture $p$
to be
\[
d_\arc(N,p) = 
   \min \{d_\arc(v,w) \st v \in \Delta_-(N), \, w \in \Delta_+(N) \}.
\]
In words, $d_\arc(N,p)$ is the length of the shortest path in $\arc(F,
p)$ from a shortest arc in the lower convex core boundary to a
shortest arc in the upper boundary.
\end{define}

\begin{theorem}\label{thm:main-qf}
Let $F$ be a surface with a preferred puncture $p$, and let $N \cong F
\times \RR$ be a quasi-Fuchsian $3$--manifold. Let $C$ be the maximal
cusp corresponding to $p$.  Then
\begin{eqnarray*}
 \frac{  d_\arc(N,p)}{450\,  \chi(F)^4 } \,-\,  \frac{1}{23\,  \chi(F)^2 } 
   & < &  \area( \bdy C \cap \core(N)) \\
   & < & 9 \, \chi(F)^2 \, d_\arc(N,p) + 
      \Big| 12 \chi(F)   \ln \abs{ \chi(F) } + 26 \chi(F) \Big|
\end{eqnarray*}
Similarly, the height of the cusp relative to a longitude satisfies
\[
\frac{  d_\arc(N,p)}{536\,  \chi(F)^4 }  \,-\, \frac{1}{27\,
  \chi(F)^2 }\: < \:  \height( \bdy C \cap \core(N)) \: < \: -3 \,
  \chi(F) \, d_\arc(N,p) + 2\ln \abs{ \chi(F) } + 5.
\]
\end{theorem}

We note that the multiplicative constants in Theorem~\ref{thm:main-qf}
are exactly the same as in Theorem~\ref{thm:main}. However, in
addition to multiplicative error, the estimates in
Theorem~\ref{thm:main-qf} contain explicit additive error. This
additive error is necessary: for example, if the limit set of $N$ is
sufficiently close to a round circle, one may have $d_\arc(N,p) =
0$. (See Lemma~\ref{lemma:prescribed-qf} for a constructive argument.)
On the other hand, $\area(\bdy C \cap \core(N)) > 0$ whenever $N$ is
not Fuchsian.

Theorem~\ref{thm:main-qf} has an interesting relation to the work of
Akiyoshi, Miyachi, and Sakuma~\cite{akiyoshi-miyachi-sakuma}.  For a
quasi-Fuchian manifold $N$, they study the scale-invariant quantity
\[
\width(\bdy C) \: := \: \height(\bdy C) / \lambda \: = \: 
\area(\bdy C) / \lambda^2 
\]
where $\lambda$ is the Euclidean length of the longitude of cusp torus
$\bdy C$.  Generalizing McShane's identity, they give an exact
expression for $\width(\bdy C)$ as the sum of an infinite series
involving the complex lengths of closed curves created by joining the
endpoints of an arc.  It seems reasonable that most of the
contribution in this infinite sum should come from the finitely many
arcs in $F$ that are shortest in $N$.  Theorem~\ref{thm:main-qf}
matches this intuition, and indeed its lower bound is proved by
summing the contributions of finitely many short arcs.

\subsection{Covers and the arc complex}

Theorem~\ref{thm:main-qf} has an interesting application to the
geometry of arc complexes, whose statement does not involve
$3$--manifolds in any way.

\begin{define}
\label{def:lifting-map}
Suppose $f \from \Sigma \to S$ is an $n$--sheeted covering map of
surfaces.  We define a relation $\pi \from \arc(S) \to \arc(\Sigma)$
as follows: $\alpha \in \pi(a)$ if and only if $\alpha$ is a component
of $f^{-1}(a)$.  In other words, $\pi(a) \subset \arc(\Sigma)$ is the
set of all $n$ lifts of $a$, which span an $(n-1)$--simplex.
\end{define}

Definition~\ref{def:lifting-map} also applies to curve complexes, with
the (inessential) difference that the number of lifts of a curve is
not determined by the degree of the cover. In this context, Rafi and
Schleimer proved that $\pi \from \curve(S) \to \curve(\Sigma)$ is a
quasi-isometric embedding~\cite{rafi-schleimer:covers}. That is, there
exist constants $K \geq 1$ and $C \geq 0$, such that for all $a,b \in
\curve^{(0)}(S)$ and for all $\alpha \in \pi(a)$, $\beta \in \pi(b)$,
we have
\[
d(a,b) \leq K \, d(\alpha, \beta) + C
\quad \mbox{and} \quad
d(\alpha,\beta) \leq K \, d(a, b) + C.
\]
The constants $K$ and $C$ depend only on $S$ and the degree of the
cover, but are not explicit.

As a consequence of Theorem~\ref{thm:main-qf}, we prove a version of
the Rafi--Schleimer theorem for arc complexes, with explicit
constants.

\begin{theorem}\label{thm:lifting}
Let $\Sigma$ and $S$ be surfaces with one puncture, and $f \from
\Sigma \to S$ a covering map of degree $n$. Let $\pi \from \arc(S) \to
\arc(\Sigma)$ be the lifting relation.  Then, for all $a,b \in
\arc^{(0)}(S)$, we have
\[
\frac{d(a,b)}{4050 \, n \, \chi(S)^6} \,  - 2 \: < \: d(\alpha, \beta) \: \leq \: d(a, b)
\]
where $\alpha \in \pi(a)$ and $\beta \in \pi(b)$.
\end{theorem}


\subsection{An outline of the arguments}
\label{subsec:outline}

The proofs of Theorems~\ref{thm:main} and~\ref{thm:main-qf} have a
decidedly elementary flavor. The primary tool that we use repeatedly
is the geometry of pleated surfaces, as developed by
Thurston~\cite{thurston:notes}. (See Bonahon~\cite{bonahon:lamination}
or Canary, Epstein, and Green~\cite{ceg:notes-on-notes} for a detailed
description.)  In our context, a pleated surface is typically a copy
of the fiber $F$ with a prescribed hyperbolic metric, immersed into
$M$ in a piecewise geodesic fashion, and bent along an ideal
triangulation of $F$.  In Sections~\ref{sec:pleated}
and~\ref{sec:cone-surface} below, we give a detailed definition of
pleated surfaces and discuss the geometry of cusp neighborhoods in
such a surface. We also study a mild generalization of pleated
surfaces, called simplicial hyperbolic surfaces, that are hyperbolic
everywhere except for a single cone point with angle at least $2\pi$.

The upper bounds of Theorems~\ref{thm:main} and~\ref{thm:main-qf} are
proved in Section~\ref{sec:upper} and~\ref{sec:upper-qf},
respectively. To sketch the argument in the fibered case, let $\tau$
be an ideal triangulation of the fiber $F$. Then $F$ can be homotoped
to a pleated surface, $F_\tau$, in which every ideal triangle is
totally geodesic.  Using lemmas in Sections~\ref{sec:pleated}
and~\ref{sec:cone-surface}, we show that the intersection $F_\tau \cap
\bdy C$ gives a closed polygonal curve about the puncture $p$, whose
length is bounded by $-6 \chi(F)$. As a result, the pleated surface
$F_\tau$ makes a bounded contribution to the area and height of $\bdy
C$. Summing up the contributions from a sequence of triangulations
that ``realize'' the monodromy $\mon$ gives the desired upper bound of
Theorem~\ref{thm:main}. The upper bound of Theorem~\ref{thm:main-qf}
uses very similar ideas; the one added ingredient is a bound on how
far a short arc in $\bdy_\pm \core(N)$ drifts when it is pulled tight,
making it geodesic in $N$.

The lower bounds on cusp area and height rely on the idea of a
geometrically controlled \emph{sweepout}. This is a degree-one map
$\Psi \from F \times [0,1]/ \mon \to M_\mon$, in which every fiber $F
\times \{ t \}$ in the domain is mapped to a piecewise geodesic
surface $F_t \subset M$, which is either pleated or simplicial
hyperbolic. The elementary construction of such a sweepout, which is
due to Thurston~\cite{thurston:notes} and Canary
\cite{canary:covering}, is recalled in Section~\ref{sec:sweepouts}.

The lower bound of Theorem~\ref{thm:main} is proved in Section
\ref{sec:lower}. We show that every piecewise geodesic surface $F_t$
in the sweepout of Section~\ref{sec:sweepouts} must contain an an arc
from cusp to cusp whose length is explicitly bounded above. As the
parameter $t$ moves around the sweepout, we obtain a sequence of arcs,
representing a walk through the $1$--skeleton $\arc^{(1)}(F,p)$, such
that each arc encountered has bounded length in $M$.  This sequence of
somewhat-short arcs in the fiber leads to a packing of the cusp torus
$\bdy C$ by shadows of somewhat-large horoballs, implying a lower
bound on $\area(\bdy C)$ and $\height(\bdy C)$.

It is worth emphasizing that the entire proof of Theorem
\ref{thm:main} is elementary in nature. In particular, this proof does
not rely on any deep results from Teichm\"uller theory, Kleinian groups,
or the coarse geometry of the curve or arc complexes.


The lower bound of Theorem~\ref{thm:main-qf} is proved in Section
\ref{sec:lower-qf}, using very similar ideas to those of Theorem
\ref{thm:main}. Once again, we have a sweepout $\Psi \from F \times 
[0,r] \to \core(N)$ by simplicial hyperbolic surfaces. Once again,
each surface $F_t$ in the sweepout contains a somewhat-short arc from
cusp to cusp, corresponding to a horoball whose shadow contributes
area to $\bdy C$.  However, we also need to know that the pleated
surfaces at the start and end of the sweepout can be chosen
arbitrarily close to $\bdy_\pm \core(N)$. This fact, written down as
Theorem~\ref{thm:core-bdy-teich} in the appendix, is the one place in
the paper where we need to reach into the non-elementary toolbox of
Kleinian groups.

Finally, in Section~\ref{sec:covers}, we prove
Theorem~\ref{thm:lifting}.  Given a cover $\Sigma \to S$, vertices $a,
b$ of $\arc(S)$, and vertices $\alpha \in \pi(a)$, $\beta \in \pi(b)$,
the upper bound on the distance $d(\alpha, \beta)$ is immediate
because disjoint arcs lift to disjoint multi-arcs.  To prove a lower
bound, we construct a quasi-Fuchsian manifold $M \cong S \times \RR$,
so that $a$ and $b$ are the unique shortest arcs on its convex core
boundaries. The hyperbolic metric on $M \cong S \times \RR$ lifts to a
quasi-Fuchsian structure on $N \cong \Sigma \times \RR$. By applying
Theorem~\ref{thm:main-qf} to both $M$ and $N$, we will bound
$d(\alpha, \beta)$ from below.

\subsection{Acknowledgments} 
This project began at the University of Warwick symposium on \emph{Low
Dimensional Geometry and Topology}, in honor of David Epstein, and
continued at the MSRI program in \emph{Teichm\"uller Theory and
Kleinian Groups}. We thank the organizers of both events for creating
such a fertile ground for collaboration.

We thank Ian Agol for numerous helpful conversations, and in
particular for contributing the key idea of Lemma
\ref{lemma:length-bound}. We thank Dick Canary and Yair Minsky for
clarifying a number of points about pleated surfaces, and for helping
us sort out the proof of Theorem~\ref{thm:core-bdy-teich}. We thank
Marc Lackenby for continually encouraging us to make our estimates
effective.


\section{Pleated surfaces and cusps}
\label{sec:pleated}

To prove the upper and lower bounds in our main theorems, we need a
detailed understanding of the geometry of pleated surfaces in a
hyperbolic $3$--manifold. In this section, we survey several known
results about pleated surfaces. We also describe the somewhat subtle
geometry of the intersection between a pleated surface and a cusp
neighborhood in a $3$--manifold $N$. The study of pleated surfaces is
continued in Section~\ref{sec:cone-surface}, where we obtain several
geometric estimates.

References for this material include Bonahon~\cite{bonahon:lamination}
and Canary, Epstein, and Green~\cite{ceg:notes-on-notes}.

\begin{define}
\label{def:lamination}
Let $S$ be a surface, as in Definition~\ref{def:arc-complex}.  A
\emph{lamination} $\lamm \subset S$ is a $1$--dimensional foliation of
a closed subset of $S$.  
\end{define}

A special case of a lamination is the union of the edges of an ideal
triangulation; this special case appears frequently in our setting.

\begin{define}
\label{def:pleated}
Let $N$ be a hyperbolic $3$--manifold, and let $S$ be a surface.  Fix
a proper map $f_0 \from S \to N$, sending punctures to cusps.  For a
lamination $\lamm \subset S$ a \emph{pleating} of $f_0$ along $\lamm$,
or a \emph{pleating map} for short, is a map $f \from S \to N$,
properly homotopic to $f_0$, such that
\begin{enumerate}
\item
$f$ maps every leaf of $\lamm$ to a hyperbolic geodesic, and
\item 
$f$ maps every component of $S \setminus \lamm$ to a totally geodesic
surface in $N$.
\end{enumerate}
We say that $f$ \emph{realizes} the lamination $\lamm$, and call its
image $f(S)$ a \emph{pleated surface}.
\end{define}

Note the existence of a pleating map places restrictions on $\lamm$;
for instance, every closed leaf of $\lamm$ must be essential and
non-peripheral in $S$.  The hyperbolic metric on $N$, viewed as a
path-metric, pulls back via $f$ to induce a complete hyperbolic metric
on $S$.  In this induced hyperbolic metric, every leaf of $\lamm$
becomes a geodesic, and the map $f \from S \to N$ becomes a piecewise
isometry, which is bent along the geodesic leaves of $f(\lamm)$.

\begin{lemma}
\label{Lem:Pleated}
Every pleated surface $f(S)$ is contained in the convex core of $N$.
\end{lemma}

\begin{proof}
Adding leaves as needed to subdivide the totally geodesic regions, we
can arrange for the complement $S \setminus \lamm$ to consist of ideal
triangles.  Fix $g$, a side of some ideal triangle of $S \setminus
\lamm$.  Thus $g$ is a bi-infinite geodesic; each end of the image
geodesic $f(g)$ either runs out a cusp of $N$ or meets a small metric
ball infinitely many times.  In either case, the lift of $f(g)$ to the
universal cover $\widetilde{N} = \HH^3$ has both endpoints at limit
points of $N$.  Since an ideal triangle in $\HH^3$ is the convex hull
of its vertices, the surface $f(S)$ is contained in $\core(N)$.
\end{proof}

In a quasi-Fuchsian manifold $N$, the components of $\bdy_\pm
\core(N)$ are themselves pleated surfaces.  In this paper, the convex
core boundaries are the only examples of pleated surfaces where the
pleating laminations are not ideal triangulations.

A foundational result is that every essential surface $S \subset N$
can be pleated along an arbitrary ideal triangulation. This was first
observed by Thurston~\cite[Chapter 8]{thurston:notes}. For a more
detailed account of the proof, see Canary, Epstein, and Green
\cite[Theorem 5.3.6]{ceg:notes-on-notes} or Lackenby~\cite[Lemma
2.2]{lackenby:surgery}.

\begin{prop}
\label{prop:pleating}
Let $N$ be a cusped orientable hyperbolic $3$--manifold. Let $f_0
\from S \to N$ be a proper, essential map, sending punctures to cusps.
Then, for any ideal triangulation $\tau$ of $S$, the map $f_0$ is
homotopic to a pleating map along $\tau$.

In other words, every ideal triangulation $\tau$ is realized by a
pleating map $f_\tau \from S \to N$.
\qed
\end{prop}

Suppose that $C \subset N$ is a horospherical cusp neighborhood in
$N$, and $f \from S \to N$ is a pleating map. Our goal is to describe
the geometry of $f(S) \cap C$.  We offer Figure~\ref{fig:pleated} as a
preview of the geometric picture.  The figure depicts a lift
$\widetilde{S}$ of a pleated surface to $\HH^3$. For a sufficiently
small horocusp $C_0 \subset C$, which lifts to horoball $H_0$ in the
figure, the intersection $f(S) \cap C_0$ is \emph{standard}, meaning
that $f^{-1}(C_0)$ is a union of horospherical cusp neighborhoods in
$S$. The intersection $f(S) \cap C$ is more complicated, because the
surface is bent along certain geodesics whose interior intersects $C
\setminus C_0$. Nevertheless, we can use the geometry of $f(S) \cap
C_0$ to find certain cusp neighborhoods contained in $f^{-1}(C)$ (in
Lemma~\ref{lemma:pleated-cusp}), and certain geometrically meaningful
closed curves in $\bdy C$ (in Lemma~\ref{lemma:projection-curve}).

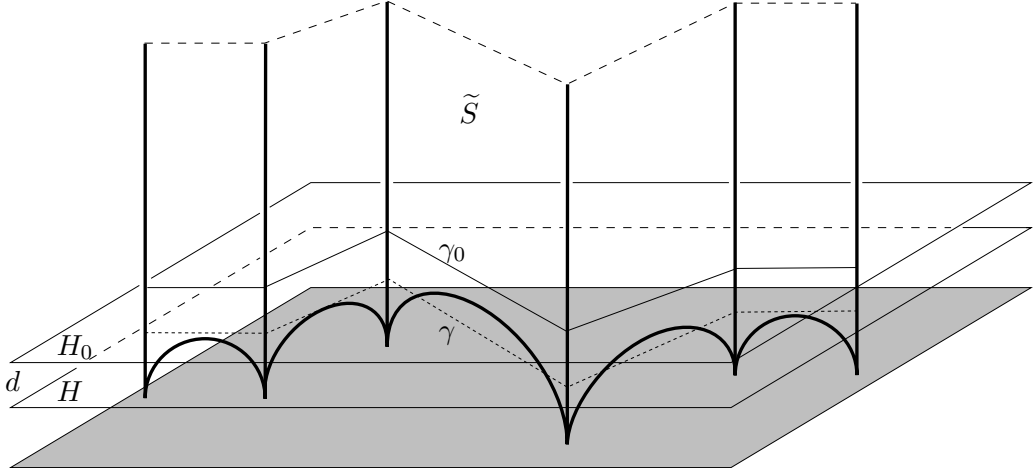
\begin{figure}
\begin{overpic}[height=2.5in]{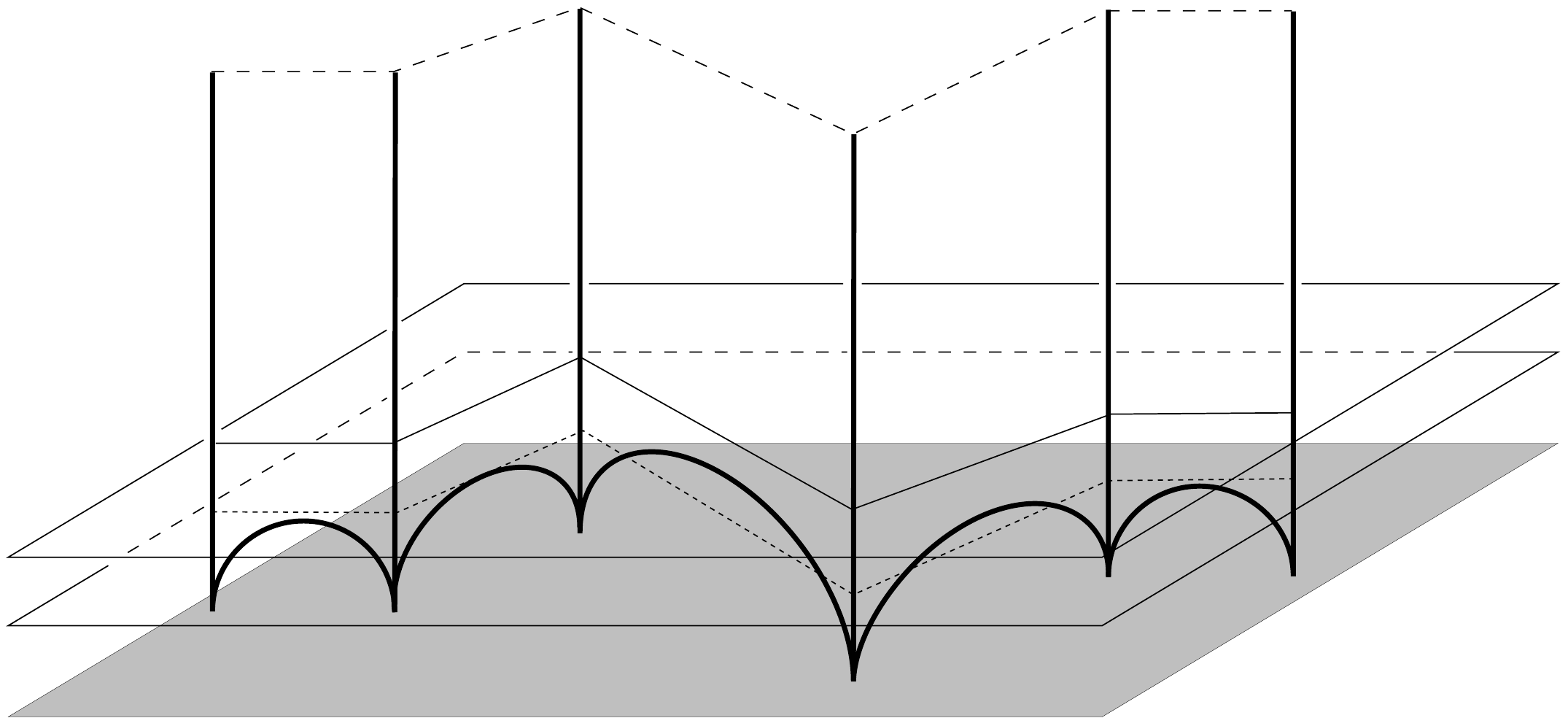}
\put(5,11.5){$H_0$}
\put(5,7){$H$}
\put(0,8){$d$}
\put(44,34){$\widetilde{S}$}
\put(42,13){$\gamma$}
\put(42,21){$\gamma_0$}
\end{overpic}
\caption{The intersection between a pleated surface and a
horocusp. The intersection with horoball $H_0$ is standard, whereas
the intersection with $H$ may contain portions of the surface bent
along geodesics whose endpoints are not in $H$. Graphic based
on a design of Agol~\cite[Figure~1]{agol:6theorem}.}
\label{fig:pleated}
\end{figure}

\begin{lemma}
\label{lemma:pleated-cusp}
Let $N$ be a cusped orientable hyperbolic $3$--manifold, with a
horocusp $C$. Let $f \from S \to N$ be a pleating map, such that $n$
punctures of $S$ are mapped to $C$. Suppose that a loop about a
puncture of $S$ is represented by a geodesic of length $\lambda$ on
$\bdy C$.  Then, in the induced hyperbolic metric on $S$, the preimage
$f^{-1} (C) \subset S$ contains horospherical cusp neighborhoods $R_1,
\ldots, R_n$ with disjoint interiors, such that
\[
\ell(\bdy R_i) \: = \: \area(R_i) \: \geq \: 
   \lambda \quad \mbox{for each $i$.}
\]
\end{lemma}

Our proof is inspired by an argument of
Agol~\cite[Theorem~5.1]{agol:6theorem}. 

\begin{proof}[Proof of Lemma~\ref{lemma:pleated-cusp}]
Without loss of generality, assume that the pleating lamination
$\lamm$ cuts $S$ into ideal triangles. (Otherwise, add more leaves to
$\lamm$.) Let $C_0 \subset C$ be a horocusp chosen sufficiently small
so that $C_0 \cap f(\lamm)$ is a union of non-compact rays into the
cusp. Then $f^{-1}(C_0) $ is a union of tips of ideal triangles in $S$
and consists of disjoint horospherical neighborhoods $R_1^0, \ldots,
R_n^0$, each mapped into $C$.

Lift $N$ to its universal cover $\HH^3$, so that $C_0$ lifts to a
horoball $H_0$ about $\infty$ in the upper half-space model. Then
$\widetilde{f(S)}$ intersects $H_0$ in vertical bands, as shown in
Figure~\ref{fig:pleated}.

Let $d$ be the distance in $N$ between $\bdy C_0$ and $\bdy C$. Since
the interior of $C$ is embedded, this means that the shortest geodesic
in $N$ from $C_0$ to $C_0$ has length at least $2d$. Since the
pleating map $f \from S \to N$ is distance-decreasing, the shortest
geodesic in $S$ from $f^{-1} (C_0)$ to itself also has length at least
$2d$. In other words, we may take a closed $d$--neighborhood of each $
R_i^0 $ and obtain a cusp neighborhood $R_i$, such that $R_1, \ldots,
R_n$ have disjoint interiors.

Consider the areas of these neighborhoods, along with their boundary
lengths.  A standard calculation in the upper half-plane model of
$\HH^2$ implies that the length of a horocycle in $S$ equals the area
of the associated cusp neighborhood. Furthermore, both quantities grow
exponentially with $d$.

On $\bdy C_0$, a Euclidean geodesic about a puncture of $S$ has length
$e^{-d} \lambda$. Since $f(S) \cap C_0$ may not be totally geodesic
(in general, it is bent, as in Figure~\ref{fig:pleated}), each curve
of $\bdy R_i^0$ has length bounded below by $e^{-d} \lambda$. These
lengths grow by $e^d$ as we take a $d$--neighborhood of $\cup_i
R_i^0$. We conclude that each component $R_i$ satisfies
\begin{equation}\label{eq:distance-d}
\ell(\bdy R_i) \: = \: \area(R_i) \: \geq \: 
    e^d \cdot e^{-d} \cdot  \lambda \: = \: \lambda.
\end{equation}

It remains to show that $f(R_i) \subset C$ for each $i$. Suppose,
without loss of generality, that $R_1^0$ is the component of
$f^{-1}(C_0)$ whose lift is mapped to the horoball $H_0$. Then, by
construction, the lift of $R_1$ is mapped into the $d$--neighborhood
of $H_0$, which is a horoball $H$ covering $C$. Thus $f(R_1) \subset
C$. Note that the containment might be strict, because $f(S)$ might be
bent along some geodesics in the region $C \setminus C_0$, as in the
middle of Figure~\ref{fig:pleated}.
\end{proof}

The argument of Lemma~\ref{lemma:pleated-cusp} also permits the
following construction, which is also important for Section
\ref{sec:upper}.

\begin{lemma}
\label{lemma:projection-curve}
Let $N$ be a cusped orientable hyperbolic $3$--manifold, with a
horocusp $C$. Let $f \from S \to N$ be a pleating map that realizes an
ideal triangulation $\tau$. Then, for each puncture $p$ of $S$ that is
mapped to $C$, there is an immersed closed curve $\gamma = \gamma(f,
p, C)$, piecewise geodesic in the Euclidean metric on $\bdy C$, with
the following properties:
\begin{enumerate}
\item 
The loop $\gamma$ is homotopic in $C$ to a loop in $f(S)$ about $p$.
\item 
The vertices of $\gamma$ lie in $\bdy C \cap f(\tau)$, and correspond
to the endpoints of edges of $\tau$ at puncture $p$.
\item 
$\ell(\gamma) = \ell(\bdy R_i) $, where $R_i \subset S$ is one of the
cusp neighborhoods of Lemma~\ref{lemma:pleated-cusp}.
\end{enumerate}
\end{lemma}

A lift of $\gamma$ to a horoball $H$ covering $C$ is shown, dotted, in
Figure~\ref{fig:pleated}.

\begin{proof}[Proof of Lemma~\ref{lemma:projection-curve}]
We may construct $\gamma$ as follows. Recall, from the proof of Lemma
\ref{lemma:pleated-cusp}, that there is a horocusp $C_0 \subset C$
such that the intersection $f(S) \cap C_0$ is standard, consisting of
tips of ideal triangles. Then $f^{-1}(C_0)$ is a disjoint union of
horospherical cusp neighborhoods. Let $R_i^0$ be the component of
$f^{-1}(C_0)$ that contains puncture $p$, and let $\gamma_0 = f(\bdy
R_i^0) \subset \bdy C_0$.

Note that the curve $\gamma_0$ is piecewise geodesic in the Euclidean
metric on $\bdy C_0$, and that it is bent precisely at the
intersection points $\bdy C_0 \cap f(\tau)$, where the triangulation
$\tau$ enters the cusp. See Figure~\ref{fig:pleated}.

We define $\gamma$ to be the projection of $\gamma_0$ to the
horospherical torus $\bdy C$. Note that if $C_0$ and $C$ are lifted to
horoballs about $\infty$ in $\HH^3$, as in Figure~\ref{fig:pleated},
this projection is just vertical projection in the upper half-space
model.

Observe that while $\gamma_0 \subset f(S)$, its projection $\gamma$
might not be contained in the pleated surface. Nevertheless, $\gamma$
is completely defined by $\gamma_0$. The vertices where $\gamma$ is
bent are contained in $f(\tau)$.

Let $d$ be the distance between $\bdy C_0$ and $\bdy C$. Then, as in
Lemma~\ref{lemma:pleated-cusp}, lengths grow by a factor of $e^d$ as
we pass from $\bdy C_0$ to $\bdy C$. Thus, by the same calculation as
in \eqref{eq:distance-d},
\[
\ell(\gamma) \: = \: e^d \cdot \ell(\gamma_0) \: = \: e^d \cdot
\area(R_i^0) \: = \: \area(R_i) \: = \: \ell(\bdy  R_i),
\]
where $R_i \supset R_i^0$ is the cusp neighborhood in $S$ that is
mapped into $C$, as in Lemma~\ref{lemma:pleated-cusp}.
\end{proof}

Our final goal in this section is to provide a universal lower bound
on the size of the cusp neighborhoods $R_i$. We do this using the
following result of Adams~\cite{adams:waist2}.

\begin{lemma}\label{lemma:waist}
Let $N$ be a non-elementary, orientable hyperbolic $3$--manifold, and
let $C$ be a maximal horocusp in $N$. (This neighborhood may
correspond to either a rank one or rank two cusp.) Let $s$ be a simple
closed curve on $\bdy C$, which forms part of the boundary of an
essential surface in $N$. Then $\ell(s) > 2^{1/4}$.
\end{lemma}

\begin{proof}
This is a consequence of a theorem of Adams
\cite[Theorem~3.3]{adams:waist2}. He proved that every parabolic
translation of the maximal cusp of any non-elementary hyperbolic
$3$--manifold has length greater than $2^{1/4}$, with exactly three
exceptions: one parabolic each in the three SnapPea census manifolds
{\tt m004}, {\tt m009}, and {\tt m015}. 

Each of the manifolds {\tt m004}, {\tt m009}, and {\tt m015} is either
a punctured torus bundle or a two-bridge knot complement. Hence the
boundary slopes of incompressible surfaces in these manifolds are
classified~\cite{floyd-hatcher, hatcher-thurston:2bridge}. In
particular, none of the three slopes shorter than $2^{1/4}$ bounds an
essential surface.
\end{proof}

As a result, we obtain

\begin{lemma}\label{lemma:pleated-cusp-size}
Let $N$ be a cusped orientable hyperbolic $3$--manifold, with a
maximal cusp $C$. Let $f(S) \subset N$ be a pleated surface, homotopic
to a properly embedded essential surface, such that $n$ punctures of
$S$ are mapped to $C$.  Then $f^{-1}(C) \subset S$ contains $n$
disjoint horospherical cusp neighborhoods $R_1, \ldots, R_n$, such
that
\[
\ell(\bdy R_i) \: = \: \area(R_i) \: > \: 2^{1/4} 
   \quad \mbox{for each $i$.}
\]
\end{lemma}

\begin{proof}
This is immediate from Lemmas~\ref{lemma:pleated-cusp} and
\ref{lemma:waist}.
\end{proof}


\section{Hyperbolic surfaces with one cone point}
\label{sec:cone-surface}

Recall from Definition~\ref{def:pleated} that every pleated surface
carries an intrinsic hyperbolic metric. In this section, we prove
several lemmas about the geometry of cusp neighborhoods and geodesic
arcs in these surfaces. These estimates are used throughout the proofs
of Theorems~\ref{thm:main} and~\ref{thm:main-qf}.

In fact, we work in a slightly more general setting: namely,
hyperbolic surfaces with a cone point, whose cone angle is at least
$2\pi$.  These singular surfaces arise in sweepouts of a hyperbolic
$3$--manifold: see Section~\ref{sec:sweepouts}. Therefore, we derive
length and area estimates for these singular surfaces, as well as
non-singular ones.

\begin{define}
\label{def:cone-surface}
A \emph{hyperbolic cone surface} is a complete metric space $S$,
homeomorphic to a surface of finite type. We require that $S$ admits a
triangulation into finitely many simplices, such that each simplex is
isometric to a totally geodesic triangle in $\HH^2$. The triangles are
allowed to have any combination of ideal vertices (which correspond to
punctures of $S$) and material vertices (which correspond to points in
$S$). The triangles are glued by isometries along their edges.

Every point of $S$ that is not a material vertex of the triangulation
thus has a neighborhood isometric to a disk in $\HH^2$. Every material
vertex $v \in S$ has a neighborhood where where the metric (in polar
coordinates) takes the form
\begin{equation}\label{eq:cone}
ds^2 = dr^2 + \sinh^2(r) \,  d\theta^2,
\end{equation}
where $0 \leq r < r_v$ and $0 \leq \theta \leq \theta_v$. Here
$\theta_v$ is called the \emph{cone angle} at $v$, and can be
computed as the sum of the interior angles at $v$ over all the
triangles that meet $v$. Note that if $\theta_v=2\pi$, equation
(\ref{eq:cone}) becomes the standard polar equation for the hyperbolic
metric in a disk. The vertices of $S$ whose cone angles are \emph{not}
equal to $2\pi$ are called the \emph{cone points} or \emph{singular
points} of $S$; all remaining points are called \emph{non-singular}.
\end{define}

If all singular points of $S$ have cone angles $\theta_v > 2\pi$,
another common name for $S$ is a \emph{simplicial hyperbolic
surface}. Simplicial hyperbolic surfaces have played an important role
in the study of geometrically infinite Kleinian groups
\cite{canary:covering, gabai:tameness-notes}.

Just as with non-singular hyperbolic surfaces, cone surfaces have a
natural geometric notion of a cusp neighborhood.

\begin{define}\label{def:equidistant-cusp}
Let $S$ be a hyperbolic cone surface, with one or more punctures, and
let $R \subset S$ be a closed set. Then $R$ is called an
\emph{equidistant cusp neighborhood} of a puncture of $S$ if the
following conditions are satisfied:
\begin{enumerate}
\item 
The interior of $R$ is homeomorphic to $S^1 \times (0, \infty)$.
\item 
There is a closed subset $Q \subset R$, whose universal cover
$\widetilde{Q}$ is isometric to a horoball in $\HH^2$. This implies
that the interior of $Q$ is disjoint from all cone points.
\item 
There is a distance $d>0$, such that $R$ is the closed
$d$--neighborhood of $Q$.
\end{enumerate}
$R$ is called a \emph{maximal cusp} if it is not a proper subset of
any larger equidistant cusp neighborhood. Equivalently, $R$ is maximal
if and only if it is not homeomorphic to $S^1 \times [0, \infty)$.
\end{define}

\begin{lemma}
\label{lemma:embeddedness}
Let $S$ be a hyperbolic surface with one cone point $v$, of angle
$\theta_v \geq 2 \pi$. Let $R \subset S$ be a non-maximal equidistant
neighborhood of a puncture of $S$. Then
\begin{enumerate}
\item
\label{item:geodesic} There is a geodesic $\alpha$ that is
shortest among all essential paths from $R$ to $R$.
\item
\label{item:eyeglass} The arc $\alpha$ is either embedded, or is
the union of a segment and a loop based at $v$. In the latter case,
there is an arbitrarily small homotopy in $S$ making $\alpha$
embedded.
\item
\label{item:disjoint} If $2\pi \leq \theta_v < 4 \pi$, and
$\beta$ is another shortest arc from $R$ to $R$, there is an
arbitrarily small homotopy making $\alpha$ and $\beta$ disjoint.
\end{enumerate}
\end{lemma}

One way to interpret Lemma~\ref{lemma:embeddedness} is as follows.
Let $p$ be a puncture of $S$. Then any shortest arc relative to a cusp
neighborhood about $p$ gives a vertex of the arc complex $\arc(S,p)$.
If there are two distinct shortest arcs, they span an edge of
$\arc(S,p)$; more generally, if there are $n$ distinct shortest arcs,
they span an $(n-1)$--simplex.  This is used in
Sections~\ref{sec:lower} and~\ref{sec:lower-qf} to construct a path in
$\arc(S,p)$.

\begin{proof}[Proof of Lemma~\ref{lemma:embeddedness}]
Let $d$ be the infimal distance in the universal cover $\widetilde{S}$ between two
different lifts of $\bdy R$; since $R$ is not a maximal cusp,
$d>0$. Furthermore, since distance to the nearest translate is an
equivariant function on $\bdy \widetilde{R}$, it achieves a
minimum. Thus there are points $x,y$ on distinct lifts of $\bdy R$,
whose distance is exactly $d$. By the Hopf--Rinow theorem for cone
manifolds~\cite[Lemma 3.7]{chk:orbifold}, the distance between $x$ and
$y$ is realized by a geodesic path $\alpha$, and every
distance-realizing path is a geodesic. This
proves~\eqref{item:geodesic}.

\medskip

Suppose that $\alpha \from [0,d] \to S$ is a unit-speed
parametrization.  The image of $\alpha$ is a graph $\Gamma =
\Image(\alpha)$.  If $\alpha$ is not embedded, then $\Gamma$ has at
least one vertex of valence larger than $2$.  In this case, we will
show that $\Gamma$ is the union of a segment and a loop based at $v$.

Let $w$ be a vertex of $\Gamma$ that has valence larger than
$2$. Consider preimages $x, y \in \alpha^{-1}(w)$, where $x < y$. Then
$[0,d]$ splits into sub-intervals
\[
I_1 = [0,x], \quad I_2 = [x,y], \quad I_3 = [y,d].
\]
Let $\alpha_i$ be the restriction of $\alpha$ to the sub-interval $I_i$. 

We claim that $\alpha_1$ must be homotopic to $\overline{\alpha_3}$
(the reverse of $\alpha_3$): otherwise, cutting out the middle segment
$\alpha_2$ would produce a shorter essential path. We also claim that
$\ell(\alpha_1) = \ell(\alpha_3)$: for, if $\ell(\alpha_1) <
\ell(\alpha_3)$, we could homotope $\alpha_3$ to $\overline{\alpha_1}$
while shortening the length of $\alpha$. In particular, the last claim
implies that $x$ and $y$ are the only preimages of $w$, and $w$ has
valence $3$ or $4$.

If $w \neq v$, then it is a non-singular point of $S$, hence a
$4$--valent vertex. This means that $\alpha_1$ meets $\alpha_3$ at a
nonzero angle. Then, exchanging $\alpha_1$ and $\overline{\alpha_3}$
by homotopy and rounding off the corner at $w$, as in Figure
\ref{fig:round-off}, produces an essential arc shorter than
$\alpha$. This is a contradiction.

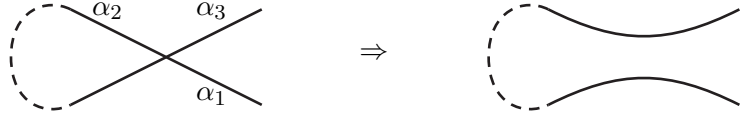
\begin{figure}[h]
\begin{overpic}{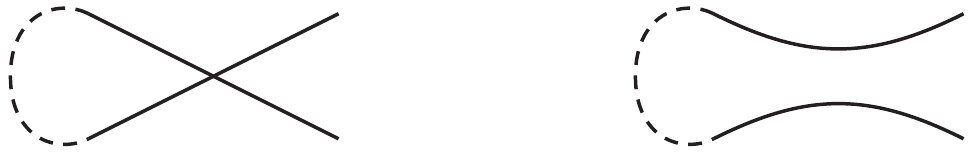}
\put(26,13.5){$\alpha_3$}
\put(12,13.5){$\alpha_2$}
\put(26,2){$\alpha_1$}
\put(48,7.5){$\Rightarrow$}
\end{overpic}
\caption{Exchanging and rounding off the arcs $\alpha_1$ and
$\alpha_3$ produces a shorter path. }
\label{fig:round-off}
\end{figure}

We may now assume that the only vertex in the interior of $\Gamma =
\Image(\alpha)$ occurs at the cone point $v$. This implies that $v$ is
not in the cusp neighborhood $R$, hence $R$ contains no singular
points and its universal cover $\widetilde{R}$ is isometric to a
horoball.  Thus, since horoballs are convex, there is a unique
shortest path from $v$ to $R$, in each homotopy class.

In particular, the homotopic arcs $\alpha_1$ and $\overline{\alpha_3}$
must coincide.  Thus $w=v$ is $3$--valent, and is the only vertex of
$\Gamma$. Hence $\alpha_2$ is an embedded loop based at $v$, and
$\alpha$ must be an ``eyeglass'' that follows $\alpha_1$ from $R$ to
$v$, runs around the loop $\alpha_2$, and returns to $R$ by retracing
$\alpha_1$. In this case, even though $\alpha$ is not embedded, one
component of the frontier of an $\varepsilon$--neighborhood of
$\Gamma$ is an embedded arc homotopic to $\alpha$. This proves
\eqref{item:eyeglass}.

For future reference, we note an important feature of eyeglass
geodesics. Suppose that $\alpha$ consists of an arc $\alpha_1$ from
$v$ to $R$ and a loop $\alpha_2$ based at $v$. Then $\alpha_1$ must be
the \emph{unique} shortest path from $v$ to $R$.  For if another
geodesic $\alpha_3$ from $v$ to $R$ has length $\ell(\alpha_3) \leq
\ell(\alpha_1)$, then $\alpha_1$ and $\alpha_3$ must be in different
homotopy classes. This means $\alpha_1 \cup \alpha_3$ is an essential
arc from $R$ to $R$, whose length is
\[
\ell(\alpha_1) + \ell(\alpha_3) \: \leq\:  2 \ell(\alpha_1) \: <\:
2\ell(\alpha_1) + \ell(\alpha_2) \: = \: \ell(a),
\]
contradicting the fact that $\alpha$ is shortest.

\medskip

For part \eqref{item:disjoint}, suppose that $\alpha$ and $\beta$ are
two distinct shortest arcs from $R$ to $R$. By statement
\eqref{item:eyeglass}, each of $\alpha$ and $\beta$ is either an
embedded arc or an eyeglass with a loop based at $v$. Suppose that
$\alpha$ and $\beta$ intersect, and let $\Gamma = \Image(\alpha) \cup
\Image(\beta)$.

If $\Gamma$ has a non-singular vertex $w$, then $w$ cuts $\alpha$ into
sub-arcs $\alpha_1, \alpha_2$ that run from $w$ to $R$. Similarly, $w$
cuts $\beta$ into sub-arcs $\beta_3, \beta_4$ from $w$ to $R$. Without
loss of generality, say that $\alpha_1$ is shortest among these four
arcs. Then at least one of $\alpha_1 \cup \beta_1$ or $\alpha_1 \cup
\beta_2$ is an essential arc from $R$ to $R$, and both of these arcs
are no longer than $\beta$. By rounding off the corner at $w$, we can
make $\alpha_1 \cup \beta_1$ or $\alpha_1 \cup \beta_2$ into an
essential arc shorter than $\beta$. This is a contradiction.

For the rest of the proof, we assume that the only vertex of $\Gamma$
is at $v$. One consequence of this assumption is that $R$ contains no
singular points.  Hence, as in the proof of \eqref{item:eyeglass},
there is a unique shortest path from $v$ to $R$ in each homotopy
class.  There are three cases: $(i)$ neither $\alpha$ nor $\beta$ is
an eyeglass, $(ii)$, $\alpha$ is an eyeglass but $\beta$ is not, and
$(iii)$ both $\alpha$ and $\beta$ are eyeglasses.

If $\alpha$ and $\beta$ are embedded arcs that intersect at $v$,
consider the valence of $v$, which must be $3$ or $4$.  If $v$ is
$3$--valent, $\alpha$ and $\beta$ must share the same path $\alpha_1 =
\beta_1$ from $R$ to $v$, then diverge. In this case, the
$\varepsilon$--neighborhood of $\Gamma$ contains disjointly embedded
arcs homotopic to $\alpha$ and $\beta$.

If $v$ is $4$--valent, let $\gamma_1, \ldots, \gamma_4$ be the four
geodesic sub-arcs of $\Gamma$ from $v$ to $R$. Since each $\gamma_i$
is the unique shortest path in its homotopy class, any combination
$\gamma_i \cup \gamma_j$ is an essential arc. Since $\alpha = \gamma_1
\cup \gamma_2$ and $\beta = \gamma_3 \cup \gamma_4$ are both shortest
arcs in $S$, every $\gamma_i$ must have the same length. But since $v$
has cone angle $\theta_v < 4\pi$, there must be two sub-arcs
$\gamma_i, \gamma_j$ that meet at an angle less than $\pi$. Thus
$\gamma_i \cup \gamma_j$ can be shortened by smoothing the corner at
$v$, contradicting the assumption that $\alpha$ and $\beta$ are
shortest.

\begin{figure}
\begin{overpic}{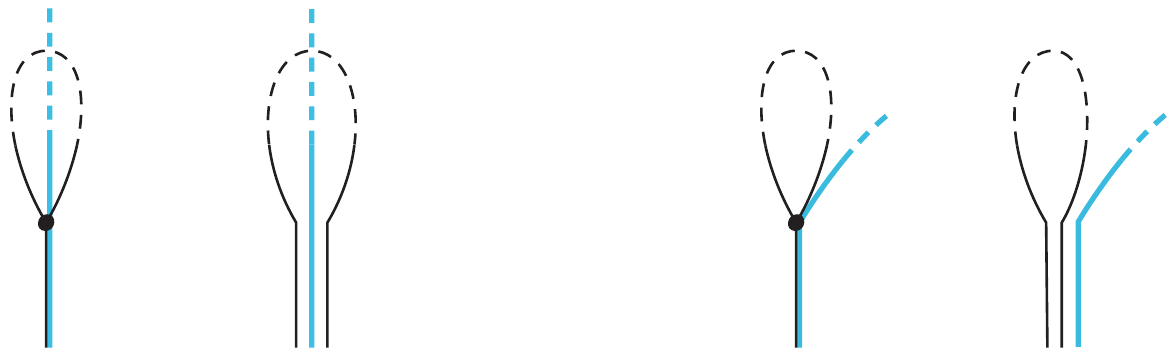}
\put(-2.5,10){$\geq \! \pi$}
\put(6,10){$\geq \! \pi$}
\put(62,10){$\geq \! \pi$}
\put(70,10){$\geq \! \pi$}
\put(-0.5,25){$\alpha$}
\put(2,29){$\beta$}
\put(22,25){$\alpha$}
\put(24.5,29){$\beta$}
\put(71,25){$\alpha$}
\put(74,21){$\beta$}
\put(92.5,25){$\alpha$}
\put(97,21){$\beta$}
\put(14,16){$\Rightarrow$}
\put(79,16){$\Rightarrow$}
\end{overpic}
\caption{An eyeglass path and an embedded arc can be made disjointly
embedded after a short homotopy. The dashed sections of $\alpha$ and
$\beta$ are schematics meant to indicate that the arcs are traveling
through a distant part of the surface, while staying disjoint.}
\label{fig:eyeglass+path}
\end{figure}
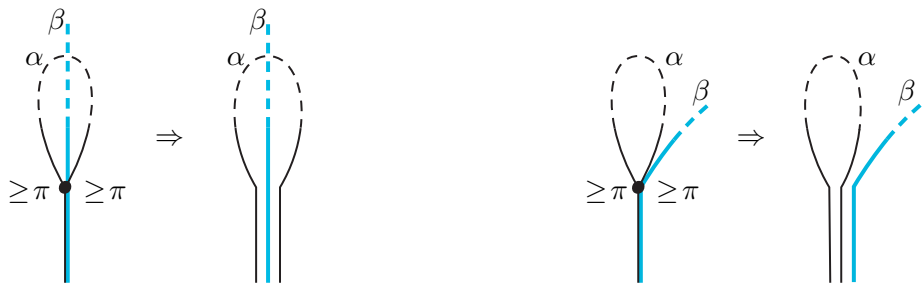

If $\alpha$ is an eyeglass, but $\beta$ is not, let $\alpha_1$ be the
sub-arc of $\alpha$ from $v$ to $R$. By the observation at the end of
part \eqref{item:eyeglass}, $\alpha_1$ is the unique shortest geodesic
from $v$ to $R$. Let $\beta_1$, $\beta_2$ be the sub-arcs of $\beta$
from $v$ to $R$, where $\ell(\beta_1) \leq \ell(\beta_2)$. If
$\alpha_1$ is distinct from $\beta_1$, then $\alpha_1 \cup \beta_1$
would be an essential path that is shorter than $\beta$ ---
contradiction. Thus $\alpha_1 = \beta_1$. In this case, homotoping
$\alpha$ to an embedded arc makes it disjoint from $\beta$. See Figure
\ref{fig:eyeglass+path}.

\begin{figure}
\begin{overpic}{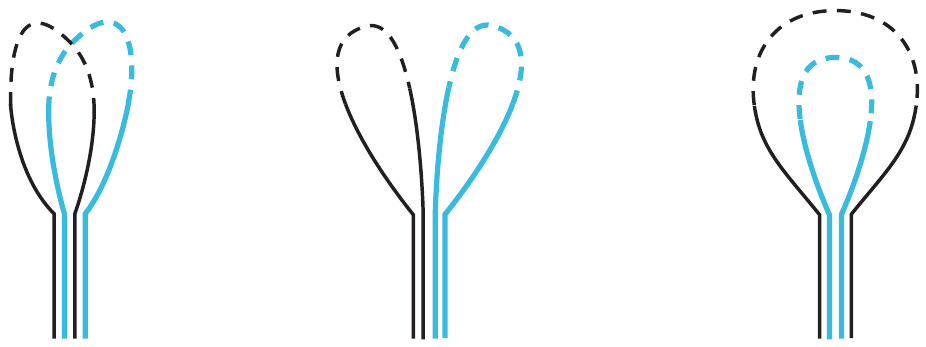}
\put(-2.5,27){$\alpha$}
\put(15,27){$\beta$}
\put(33,27){$\alpha$}
\put(57,27){$\beta$}
\put(78,27){$\alpha$}
\put(89,27){$\beta$}
\put(-2,12){$\geq \! \pi$}
\put(11,12){$\geq \! \pi$}
\put(49.5,12){$\geq \! \pi$}
\put(37,12){$\geq \! \pi$}
\put(49.5,12){$\geq \! \pi$}
\put(81,12){$\geq \! \pi$}
\put(93.5,12){$\geq \! \pi$}
\end{overpic}
\caption{Two eyeglass paths that share the same stem can be made
disjointly embedded after a short homotopy. Shown are the embedded
versions of $\alpha$ and $\beta$, in the three possible interleaving
configurations at vertex $v$.}
\label{fig:eyeglass+eyeglass}
\end{figure}
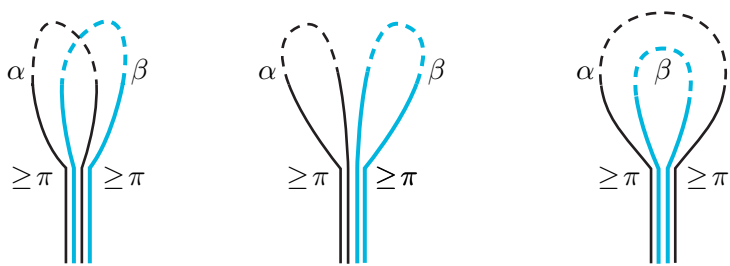

Finally, if each of $\alpha$ and $\beta$ is an eyeglass, the
observation at the end of part \eqref{item:eyeglass} implies that each
of $\alpha$ and $\beta$ must contain the unique shortest geodesic from
$v$ to $R$. Thus each of $\alpha$ and $\beta$ consists of the same
geodesic arc $\gamma_1$ from $v$ to $R$, as well as a loop based at
$v$. Figure~\ref{fig:eyeglass+eyeglass} shows that the
$\varepsilon$--neighborhood of $\Gamma = \Image(\alpha) \cup
\Image(\beta)$ contains disjointly embedded paths representing
$\alpha$ and $\beta$.
\end{proof}

In the case where $S$ is a non-singular surface, we have a stronger
version of Lemma~\ref{lemma:embeddedness}: not only are shortest arcs
disjoint, but nearly-shortest arcs must be disjoint as well.

\begin{lemma}
\label{lemma:shortish-disjoint}
Let $S$ be a punctured hyperbolic surface, and let $R$ be a
horospherical neighborhood about one puncture. Let $\alpha$ and
$\beta$ be distinct geodesic arcs from $R$ to $R$.  If $\alpha$ and
$\beta$ intersect, then there is a third geodesic arc $\gamma$,
satisfying
\[
\ell(\gamma) \: \leq \: \max \{ \ell(\alpha), \, \ell(\beta) \} 
    - \ln(2).
\]
Here, all lengths are measured relative to the cusp neighborhood $R$.
\end{lemma}

In practice, we use the contrapositive statement: if both $\alpha$ and
$\beta$ are at most $\ln (2)$ longer than the shortest geodesic from
$R$ to $R$, then they must be disjoint.

\begin{proof}[Proof of Lemma~\ref{lemma:shortish-disjoint}]
If we change the size of the cusp neighborhood $R$, then all geodesic
arcs from $R$ to $R$ have their lengths changed by the same additive
constant. Thus, without loss of generality, we may assume that $R$ is
small enough so that all intersections between $\alpha$ and $\beta$
happen outside $R$.  As $S$ has no cone points, $\alpha$ and $\beta$
meet transversely.

Orient both $\alpha$ and $\beta$. If we cut the geodesic $\alpha$
along its intersection points with $\beta$, we obtain a collection of
segments. Let $\alpha_1$ and $\alpha_2$ be the first and last such
segments, respectively, along an orientation of $\alpha$.  That is,
$\alpha_1$ (respectively $\alpha_2$) is the sub-arc of $\alpha$ from
$\bdy R$ to the first (last) point of intersection with $\beta$.
Similarly, let $\beta_1$ and $\beta_2$ be the first and last segments
of $\beta$, along an orientation of $\beta$.

Assume, without loss of generality, that $\alpha_1$ is shortest among
the segments $\alpha_1, \alpha_2, \beta_1, \beta_2$.  Set $v =
\alpha_1 \cap \beta$.  Then the vertex $v$ cuts $\beta$ into sub-arcs
$\beta_3$ and $\beta_4$, such that $\beta_1 \subset \beta_3$ and
$\beta_2 \subset \beta_4$.  Without loss of generality, we may also
assume that $\beta_3$ (rather than $\beta_4$) is the sub-arc of
$\beta$ that meets $\alpha_1$ at an angle of at most $\pi/2$.

Note that $\alpha_1 \cup \beta_3$ is an embedded arc, because (by
construction) $\alpha_1$ only intersects $\beta$ at the vertex
$v$. Furthermore, $\alpha_1 \cup \beta_3$ is topologically essential
(otherwise, one could homotope $\beta$ to reduce its length).  The
hypothesis that $\alpha_1$ is shortest among $\alpha_1, \alpha_2,
\beta_1, \beta_2$ implies that
\[
\ell(\alpha_1 \cup \beta_3) \: \leq \:  \ell(\beta_2 \cup \beta_3) \:
\leq \:  \ell(\beta_4 \cup \beta_3) \: = \: \ell(\beta).
\]

Let $\gamma$ denote the geodesic from $R$ to $R$ in the homotopy class
of $\alpha_1 \cup \beta_3$.  Then the geodesic extensions of
$\alpha_1$, $\beta_3$, and $\gamma$ form a $2/3$ ideal triangle, with
angle $\theta \leq \pi/2$ at the material vertex $v$. Then,
\cite[Lemma A.3]{cfp:tunnels} gives
\[
\ell(\gamma) \: = \: \ell(\alpha_1 \cup \beta_3) +  \ln \left(
\tfrac{1 - \cos \theta}{2}    \right) \: \leq \: \ell(\beta) + \ln
\left( \tfrac{1}{2} \right),
\]
as desired.
\end{proof}

Next, we consider the area of equidistant cusp neighborhoods in $S$.

\begin{lemma}
\label{lemma:expansion-rate}
Let $S$ be a simplicial hyperbolic surface with at most one singular
point. Let $R_1$ and $R_2$ be embedded equidistant neighborhoods of
the same puncture of $S$, such that $R_1 \subset R_2 \subset S$, and
$d$ is the distance between $\bdy R_1$ and $\bdy R_2$. Then
\[
\area(R_2) \geq e^d \, \area(R_1).
\]
\end{lemma}

\begin{proof}
Let $Q \subset R_1$ be a cusp neighborhood isometric to the quotient
of a horoball. Then there is a number $ m>0$, such that for all $x \in
[0,m]$, the closed $x$--neighborhood of $Q$ is an equidistant cusp
neighborhood $R(x)$.  In particular, $R_1 = R(x_1)$ and $R_2 =
R(x_2)$, where $x_2 = x_1 + d$.  We shall explore the dependence of
$\area(R(x))$ on the parameter $x$.

Let $v$ be the singular point of $S$. (If $S$ is non-singular, let $v$
be an arbitrary point of $S \setminus Q$, and consider it a cone point
of cone angle $2\pi$.) Let $x_v$ be the distance from $v$ to $\bdy
Q$. Then, for $x < x_v$, $R(x)$ is a non-singular neighborhood of a
cusp, itself the quotient of a horoball. As mentioned in the proof
of Lemma~\ref{lemma:pleated-cusp}, the area of a horospherical cusp
grows exponentially with distance. In symbols,
\[
\area(R(x)) = e^x\, \area(Q) \quad \mbox{if} \quad x \leq x_v.
\]

For $x > x_v$, the cusp neighborhood $R(x)$ can be constructed from a
horoball and a cone. More precisely: take a horospherical cusp, and
cut it along a vertical slit of length $r = x-x_v$. Then, take a cone
of radius $r$ and angle $\theta = \theta_v - 2\pi$, and cut it open
along a radius. Gluing these pieces together along the slits produces
$R(x)$. The area of a hyperbolic cone with radius $r$ and angle
$\theta$ can be computed as $2\theta \sinh^2(r/2)$. Adding this to the
area of a horospherical cusp, we obtain
\[
\area(R(x)) = \left\{
\begin{array}{l l}
e^x\, \area(Q) & \mbox{if} \quad x < x_v, \\
e^x\, \area(Q) + 2 \theta \sinh^2((x-x_v)/2) & \mbox{if} \quad x \geq x_v.
\end{array} \right.
\]

To complete the proof, it suffices to check that the function $f(x) =
\sinh^2 (x/2)$ grows super-exponentially for $x\geq 0$:
\begin{eqnarray*}
\left( \sinh  \frac{x+d}{2} \right)^2 &=& \left( \sinh \frac{x}{2}\,
  \cosh \frac{d}{2} + \cosh \frac{x}{2} \, \sinh \frac{d}{2} \right)^2\\ 
&>&  \left( \sinh \frac{x}{2}\,  \cosh \frac{d}{2} + \sinh \frac{x}{2} \, \sinh \frac{d}{2} \right)^2\\
&=& \left(e^{d/2} \sinh \frac{x}{2} \right)^2 \\
&=& e^d \sinh^2 \left( \frac{x}{2} \right).
\end{eqnarray*}
Thus, since the area of $R(x)$ is the sum of two functions, each of
which grows at least exponentially with $x$, it follows that
$\area(R(x+d)) \geq e^d \, \area(R(x))$.
\end{proof}

\begin{lemma}\label{lemma:circle-packing-bound}
Let $S$ be a simplicial hyperbolic surface with at most one singular
point. Let $R_{\max} \subset S$ be a maximal cusp neighborhood of a
puncture of $S$. Then
$$\area(R_{\max}) \leq -2\pi \chi(S).$$
Furthermore, if $S$ is non-singular, then
$$\area(R_{\max}) \leq -6 \chi(S).$$
\end{lemma}

\begin{proof}
Let $\theta_v \geq 2\pi$ be the cone angle at the singular point. (As
above, we take $\theta_v = 2\pi$ if the surface $S$ is non-singular.)
Then, by the Gauss--Bonnet theorem~\cite[Theorem 3.15]{chk:orbifold},
\begin{equation}\label{eq:gauss-bonnet}
\area(R_{\max}) \: \leq \: \area(S) \: = \:  - 2\pi \chi(S) + [2\pi - \theta_v]  \: \leq \:  - 2\pi \chi(S).
\end{equation}

If $S$ is a non-singular surface, then horosphere packing estimates of
B\"or\"oczky~\cite{boroczky} imply that at most $3/\pi$ of the area of
$S$ can be contained in the cusp neighborhood $R_{\max}$. Thus the
above estimate improves to $\area(R_{\max}) \leq -6 \chi(S)$.
\end{proof}

\begin{remark} 
By the Gauss--Bonnet theorem expressed in equation
\eqref{eq:gauss-bonnet}, the area of $S$ decreases as the cone angle
$\theta_v$ increases from $2\pi$. Thus it seems reasonable that the
area of a maximal cusp $R_{\max}$ would also decrease as $\theta_v$
increases from $2\pi$. If this conjecture is true, then the estimate
$\area(R_{\max}) \leq -6 \chi(S)$ would hold for singular hyperbolic
surfaces as well as non-singular ones.
\end{remark}

\begin{lemma}
\label{lemma:length-bound}
Let $S$ be a simplicial hyperbolic surface with at most one singular
point. Let $R \subset S$ be an embedded equidistant neighborhood of a
puncture of $S$. Then there exists a geodesic arc $\alpha$ from $R$ to
$R$, satisfying
\[
\ell(\alpha) \leq 2 \ln \abs{ 2\pi \,  \chi(S) / \area(R) }.
\]
Furthermore, if $S$ is non-singular, then 
\[
\ell(\alpha) \leq 2 \ln \abs{ 6 \,  \chi(S) / \area(R) }.
\]
\end{lemma}

\begin{proof}
By Lemma~\ref{lemma:embeddedness}, there is a geodesic arc $\alpha$
that is shortest among all essential arcs from $R$ to $R$. Let
$R_{\max}$ be the maximal cusp neighborhood containing $R$. Then, by
construction, $R_{\max}$ meets itself at the midpoint of
$\alpha$. Thus the distance from $\bdy R$ to $\bdy R_{\max}$ is $d =
\ell(\alpha)/2$.  By Lemma~\ref{lemma:expansion-rate}, this implies
\[
e^{d}\,  \area(R) \: \leq \: \area(R_{\max}),
\]
which simplifies to
\[
\ell(\alpha) \: = \: 2d \: \leq \: 2 \ln \left( \area(R_{\max} )/
\area(R) \right) .
\]
Substituting the bound on $\area(R_{\max})$ from Lemma
\ref{lemma:circle-packing-bound} completes the proof.
\end{proof}


\section{Upper bound: fibered manifolds}\label{sec:upper}

In this section, we prove the upper bounds of Theorem
\ref{thm:main}. We begin with a slightly simpler statement:

\begin{theorem}
\label{thm:upper}
Let $F$ be an orientable hyperbolic surface with a preferred puncture
$p$, and let $\psi \from F \to F$ be an orientation-preserving,
pseudo-Anosov homeomorphism such that $\psi(p)=p$. In the mapping
torus $M_\psi$, let $C$ be the maximal 
cusp that corresponds to $p$. Then
\[
\area(\bdy C) \: \leq \: 9 \, \chi(F)^2 \, d_{\arc} (\psi)
\quad \mbox{ and } \quad
\height(\bdy C) \: < \: -3 \, \chi(F) \, d_{\arc} (\psi).
\]
\end{theorem}

Theorem~\ref{thm:upper} differs from the upper bound of Theorem
\ref{thm:main} in two relatively small ways. First, Theorem
\ref{thm:upper} restricts attention to monodromies that fix the
puncture $p$.  (Given an arbitrary pseudo-Anosov $\mon$, one can let
$\psi$ be the smallest power of $\mon$ such that $\psi(p)=p$.)
Second, Theorem~\ref{thm:upper} estimates cusp area and height in
terms of the translation distance $d_\arc(\psi)$, rather than the
stable translation distance $\dbar_\arc(\psi)$. We shall see at the
end of the section that this simpler statement quickly implies the
upper bound of Theorem~\ref{thm:main}.

\begin{proof}[Proof of Theorem~\ref{thm:upper}]
The proof involves a direct construction. Suppose that $d_{\arc}
(\psi) = k$. Then, by Definition~\ref{def:trans-distance}, there is a
vertex $a_0 \in \arc^{(0)}(F,p)$, that is an isotopy class of arc
in $F$ meeting the puncture $p$, so that $d_\arc(a_0, \psi(a_0)) = k$. Fix
a geodesic segment in $\arc^{(1)}(F,p)$ with vertices
\[
a_0, a_1, \ldots, a_k = \psi(a_0).
\] 
By Definition~\ref{def:arc-complex}, the arcs representing $a_{i-1}$
and $a_i$ are disjoint. Thus, for every $i=1,\ldots,k$, we can choose
an ideal triangulation $\tau_i$ of $F$ that contains $a_{i-1}$ and
$a_i$.  In the arc complex $\arc(F,p)$, the geodesic segment from
$a_0$ to $a_k$ extends to a bi-infinite, $\psi$--invariant, piecewise
geodesic. Similarly, the sequence of ideal triangulations $\tau_1,
\ldots, \tau_k$ extends to a bi-infinite sequence of triangulations in
which $\tau_{i+k} = \psi(\tau_i)$.

To prove the upper bounds on cusp area and height, it is convenient to
work with the infinite cyclic cover of $M_\psi$. This is a hyperbolic
manifold $N_{\psi} \cong F \times \RR$, in which the torus cusps of
$M$ lift to annular, rank one cusps.  Let $A \subset N_{\psi}$ be the
lift of $\bdy C$ that corresponds to the puncture $p$ of $F$.

We choose geodesic coordinates for the Euclidean metric on $A \cong
S^1 \times \RR$, in which the non-trivial circle (a longitude about
$p$) is horizontal, and the $\RR$ direction is vertical. We also
choose an orientation for every arc $a_i$, such that the oriented edge
$a_i$ points into the preferred puncture $p$. In the $3$--manifold
$N_{\psi}$, the edge $a_i$ is homotopic to a unique oriented
geodesic. Given our choices, every arc $a_i$ has an associated
\emph{height} $h(a_i)$, namely the vertical coordinate of the point of
$A$ where the oriented geodesic representing $a_i$ enters the
cusp. For simplicity, we may assume that $h(a_0) = 0$ and $h(a_k) =
h(\psi(a_0)) > 0$. See Figure~\ref{fig:annulus-heights}.

\begin{figure}
\begin{center}
\begin{overpic}{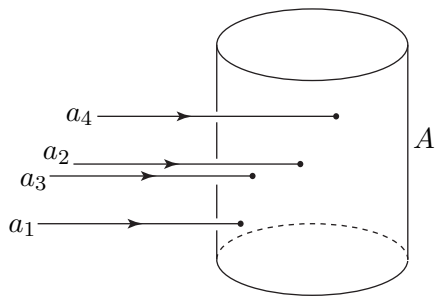}
\put(-6,18){$a_1$}
\put(-3,30){$a_3$}
\put(3,37){$a_2$}
\put(9,47){$a_4$}
\put(100,40){$A$}
\end{overpic}
\caption{The edges $a_i$ enter the cusp annulus $A \subset N_{\psi}$
at well-defined heights.}
\label{fig:annulus-heights}
\end{center}
\end{figure}

To estimate distances on the annulus $A$, we place many copies of the
fiber $F$ into pleated form.  That is, for each ideal triangulation
$\tau_i$, $i \in \ZZ$, Proposition~\ref{prop:pleating} implies that
the fiber $F \times \{0 \} \subset N_\psi = F \times \RR$ can be
homotoped in $N_\psi$ to a pleated surface $F_i$ realizing the
triangulation $\tau_i$. Recall that every ideal triangle of $\tau_i$
is totally geodesic in $F_i$.

For each pleated surface $F_i$, Lemma~\ref{lemma:projection-curve}
gives a possibly self-intersecting, piecewise geodesic closed curve
$\gamma_i \subset A$, which is homotopic to a loop about the puncture
$p$. The curve $\gamma_i$ is not necessarily contained in $F_i \cap
A$, but we do know that the vertices where it bends are the endpoints
of edges of $\tau_i$ meeting the annulus $A$. See Figure
\ref{fig:pleated} for a review.

\begin{lemma}\label{lemma:horocycle-length}
Each piecewise linear closed curve $\gamma_i \subset A$, determined by
the pleated surface $F_i$, has length $\: \ell(\gamma_i) \leq -6 \,
\chi(F)$.
\end{lemma}

\begin{proof}
Lemma~\ref{lemma:projection-curve} states that $\ell(\gamma_i) =
\ell(\bdy R_i) = \area(R_i)$, where $R_i \subset F_i$ is an embedded
horospherical neighborhood of the puncture $p$. By Lemma
\ref{lemma:circle-packing-bound}, $\area(R_i) \leq -6 \, \chi(F)$.
\end{proof}

See Agol~\cite[Theorem 5.1]{agol:6theorem} or Lackenby~\cite[Lemma
3.3]{lackenby:surgery} for a very similar statement, on which Lemma
\ref{lemma:horocycle-length} is based.

Applying Lemma~\ref{lemma:horocycle-length} to the pleated surface
$F_i$ gives a height estimate.

\begin{lemma}\label{lemma:height-segment}
The heights of consecutive arcs satisfy 
$\: \abs{h(a_i) - h(a_{i-1})} < -3 \, \chi(F)$.
\end{lemma}

\begin{proof}
By construction, the arcs $a_{i-1}$ and $a_i$ have endpoints at the
puncture $p$. Additionally, the geodesic representatives of both arcs
are contained in the triangulation $\tau_i$ along which $F_i$ is bent.
Thus the piecewise geodesic closed curve $\gamma_i$, containing the
vertices at the forward endpoints of $a_{i-1}$ and $a_i$, must visit
heights $h(a_{i-1})$ and $h(a_i)$. See Figure~\ref{fig:cusp-zigzag},
left.  Since $\ell(\gamma_i) \leq -6 \, \chi(F)$, and this closed
curve covers the distance between heights $h(a_{i-1})$ and $h(a_i)$ at
least twice, we conclude that
\[
\abs{h(a_i) - h(a_{i-1})} < -3 \, \chi(F).
\]
The inequality is strict because $\gamma_i$ must also travel around a
horizontal loop in $A$.
\end{proof}

\begin{figure}
\psfrag{a}{$a_{i-1}$}
\psfrag{b}{$a_i$}
\psfrag{g}{$\gamma_i$}
\psfrag{h0}{$h(a_{i-1})$}
\psfrag{h1}{$h(a_i)$}
\begin{center}
\begin{overpic}{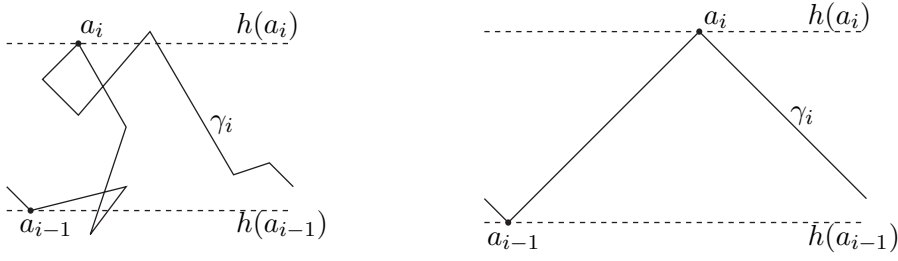}
\put(2,1){$a_{i-1}$}
\put(27,1){$h(a_{i-1})$}
\put(56,-0.5){$a_{i-1}$}
\put(93,-0.5){$h(a_{i-1})$}
\put(9,24){$a_i$}
\put(27,24){$h(a_i)$}
\put(81,25){$a_i$}
\put(93,25){$h(a_i)$}
\put(24,13){$\gamma_i$}
\put(91,14){$\gamma_i$}
\end{overpic}
\caption{Left: the polygonal closed curve $\gamma_i \subset A
$. Right: the shape of $\gamma_i$ that maximizes the area of the band
$B_i$ between heights $h(a_{i-1})$ and $h(a_i)$.}
\label{fig:cusp-zigzag}
\end{center}
\end{figure}

Lemma~\ref{lemma:horocycle-length} also leads to an area estimate.

\begin{lemma}
\label{lemma:area-segment}
Let $B_i \subset A$ be the band whose boundary consists of horizontal
circles at heights $h(a_{i-1})$ and $h(a_i)$. Then $\area(B_i) \leq 9
\, \chi(F)^2$.
\end{lemma}

\begin{proof}
As in the proof of Lemma~\ref{lemma:height-segment}, we study the
piecewise geodesic closed curve $\gamma_i \subset A$. Since $\gamma_i$
meets both $a_{i-1}$ and $a_i$, it must meet both boundary components
of $B_i$. The goal is to determine the shape of $\gamma_i$ that allows
the largest possible area for $B_i$.

Without loss of generality, we may assume that $\gamma_i$ contains
exactly two geodesic segments connecting the two boundary circles of
$B_i$: otherwise, one can straighten $\gamma_i$ while stretching
$B_i$. Such a piecewise-linear loop consisting of two segments splits
$B_i$ into two isometric triangles: one triangle below $\gamma_i$, and
the other triangle above $\gamma_i$. See Figure~\ref{fig:cusp-zigzag},
right.

At this point, we have reduced to the classical calculus problem of
building a triangular corral adjacent to a river.  As is well-known,
the optimal shape for $\gamma_i$ is one where the two segments have
the same length and meet at right angles. By
Lemma~\ref{lemma:horocycle-length}, the total length of these two
equal segments is at most $-6 \, \chi(F)$. Therefore, the maximum
possible area for $B_i$ is $9 \, \chi(F)^2$.
\end{proof}

We can now complete the proof of Theorem~\ref{thm:upper}.  A
fundamental domain for the torus $\bdy C$ is the portion of $A$
between height $h(a_0) = 0$ and height $h(a_k) = h(\psi(a_0))$. This
fundamental domain is contained in $B_1 \cup \ldots \cup B_k$. (The
containment might be strict, since there is no guarantee that the
sequence $h(a_i)$ is monotonically increasing; see Figure
\ref{fig:annulus-heights}.) Thus, by Lemma~\ref{lemma:area-segment},
\[
\area(\bdy C) \: \leq \: \sum_{i=1}^{k} \area(B_i) \: \leq \: 9 k \,
\chi(F)^2.
\]
Similarly, by Lemma~\ref{lemma:height-segment},
\[
\height(\bdy C) \: = \:  h(a_k) - h(a_0) \: \leq \: 
\sum_{i=1}^{k} \abs{h(a_i) - h(a_{i-1})} \:< \: -3k \, \chi(F).
\]
Recalling that $d_{\arc}(\psi)= k$ completes the proof. 
\end{proof}

\begin{corollary}
\label{cor:avg-upper}
Let $F$ be an orientable hyperbolic surface with a preferred puncture
$p$, and let $\psi \from F \to F$ be an orientation-preserving,
pseudo-Anosov homeomorphism such that $\psi(p)=p$. In the mapping
torus $M_\psi$, let $C$ be the maximal 
cusp that corresponds to $p$. Then
\[
\area(\bdy C) \: \leq \: 9 \, \chi(F)^2 \, \dbar_{\arc} (\psi)
\quad \mbox{ and } \quad
\height(\bdy C) \: < \: -3 \, \chi(F) \, \dbar_{\arc} (\psi).
\]
\end{corollary}

Corollary~\ref{cor:avg-upper} differs from Theorem~\ref{thm:upper} in
that $d_\arc$ has been replaced by $\dbar_\arc$. Since
$\dbar_\arc(\psi) \leq d_\arc(\psi)$ by triangle inequalities, the
statement of Corollary~\ref{cor:avg-upper} is slightly sharper.

\begin{proof}[Proof of Corollary~\ref{cor:avg-upper}]
Let $n \geq 1$. The maximal cusp $C$ of $M_\psi$ lifts to an embedded
horocusp in $M_{\psi^n}$, whose area is $n \cdot \area(\bdy
C)$. Applying Theorem~\ref{thm:upper} to $M_{\psi^n}$, we obtain
\[
n \cdot \area(\bdy C) \: \leq \: 9 \, \chi(F)^2 \, d_{\arc} (\psi^n).
\]
Thus
\begin{eqnarray*}
 \area(\bdy C) 
&\leq& 9 \, \chi(F)^2 \, \inf_{n \geq 1} \frac{ d_{\arc} (\psi^n)}{n} \\
&\leq& 9 \, \chi(F)^2 \, \liminf_{n \to \infty} \frac{ d_{\arc} (\psi^n)}{n} \\
&\leq& 9 \, \chi(F)^2 \, \, \dbar_{\arc} (\psi).
\end{eqnarray*}
The identical calculation goes through for $\height(\bdy C)$.
\end{proof}

\begin{proof}[Proof of Theorem~\ref{thm:main}, upper bound]
Let $\mon \from F \to F$ be a pseudo-Anosov homeomorphism, and let
$\psi = \mon^n$ be the smallest power of $\mon$ that fixes the
puncture $p$. Let $C$ be the maximal cusp of $M_\mon$ corresponding to
$p$. Then $C$ lifts to a (not necessarily maximal) horocusp $C'
\subset M_\psi$, which is a one-sheeted cover of $C$.
Corollary~\ref{cor:avg-upper} gives upper bounds on the area and
height of the maximal cusp of $M_\psi$, which implies upper bounds on
the (possibly smaller) area and height of $C$.
\end{proof}


\section{Upper bound: quasi-Fuchsian manifolds}
\label{sec:upper-qf}
In this section, we prove the upper bounds of Theorem
\ref{thm:main-qf}. The proof strategy is nearly the same as the proof
of Theorem~\ref{thm:upper}, with the quasi-Fuchsian manifold $N \cong
F \times \RR$ playing the same role as $N_\psi $ in
the previous section. The main geometric difference is that
$\core(N_\psi)$ is the whole manifold, whereas $\core(N)$ has finite
volume and finite cusp area.

Let $C \subset N$ be the maximal cusp corresponding to the puncture
$p$ of $F$. As in Section~\ref{sec:upper}, we choose geodesic
coordinates on $A = \bdy C \cong S^1 \times \RR$, in which the $\RR$
direction is vertical. Since $\core(N)$ is convex, the intersection
$\core(N) \cap A$ must be a compact annulus whose boundary is a pair
of horizontal circles. We choose the orientation on $\RR$ so that
$\bdy_+ \core(N)$ is higher than $\bdy_- \core(N)$. Then every
oriented essential arc $a_i \subset N$ whose forward endpoint is at
$C$ has a well-defined height $h(a_i)$, namely the vertical coordinate
of the point on $A$ where the geodesic homotopic to $a_i$ enters the
cusp. See Figure~\ref{fig:annulus-heights}.

Following Definition~\ref{def:distance-qf}, let $\Delta_\pm(N)$ be the
simplex in $\arc(F,p)$ consisting of all shortest arcs from $p$ to $p$
in $\bdy_\pm \core (N)$. Let $a_0 \in \Delta_-(N)$ and $a_k \in
\Delta_+(N)$ realize the distance between these simplices, so that
\[ 
k \: = \:  d_\arc(a_0, a_k) \: = \:  d_\arc(N,p). 
\]
Since $a_0$ has both of its endpoints at $p$, we may choose the
orientation on $a_0$ so that the point where $a_0$ enters the cusp is
the lower of the two endpoints. Similarly, we choose the orientation
on $a_k$ so that the point where $a_k$ enters the cusp is the higher
of the two endpoints.

\begin{lemma}
\label{lemma:upper-qf-start}
Let $B \subset A$ be the compact annular band whose boundary consists
of horizontal circles at heights $h(a_0)$ and $h(a_k)$. Then
\[
\area(B) \: \leq \: 9 \, \chi(F)^2 \, d_\arc(N, p)
\quad \mbox{ and } \quad
|h(a_k) - h(a_0)| \: < \: -3 \, \chi(F) \, d_\arc(N, p).
\]
\end{lemma}

\begin{proof}
This follows from the results of Section~\ref{sec:upper}.  Let $a_0,
a_1, \ldots, a_k$ be the vertices of a geodesic in $\arc^{(1)}(F,p)$
between $a_0$ and $a_k$. Then, for every $i$, let $B_i \subset A$ be
the annular band whose boundary consists of horizontal circles at
heights $h(a_{i-1})$ and $h(a_i)$. By
Lemmas~\ref{lemma:height-segment} and~\ref{lemma:area-segment},
\[
\area(B_i) \: \leq \: 9 \, \chi(F)^2
\quad \mbox{ and } \quad
 | h(a_i) - h(a_{i-1})| \: < \: -3 \, \chi(F)  .
\]
Adding up these estimates as $i$ ranges from $1$ to $k$ gives the
result.
\end{proof}

To prove the upper bounds of Theorem~\ref{thm:main-qf}, it remains to
estimate the area and height of the part of $\core(N) \cap A$ that is
\emph{not} contained in the band $B$. To make this region more
precise, define $h_\pm(N)$ to be the vertical coordinate of the circle
$\bdy_\pm \core(N) \cap A$. Then we may define $B_- = B_-(N)$ to be
the band whose boundary consists of horizontal circles at heights
$h_-(N)$ and $h(a_0)$, and similarly $B_+ = B_+(N)$ to be the band
between heights $h(a_k)$ and $h_+(N)$. Note that the orientations of
$a_0$ and $a_k$ were chosen precisely so as to minimize the size of
$B_-(N)$ and $B_+(N)$, respectively.

Recall that $\bdy_- \core(N)$ is an intrinsically hyperbolic surface,
pleated along a lamination. The arc $a_0$ has a geodesic
representative in $\bdy_- \core(N)$; in fact, by definition this
geodesic is shortest in $\bdy_- \core(N)$ among all arcs from $p$ to
$p$. Then $h_-(N)$ is the height at which the geodesic representative
of $a_0$ in $\bdy_- \core(N)$ enters the cusp $C$, and $h(a_0)$ is the
height at which the geodesic representative of $a_0$ in $N$ enters the
cusp $C$. The difference $\abs{h(a_0) - h_-(N)}$ is the height of
$B_-(N)$. We control $\abs{h(a_0) - h_-(N)}$ via the following
proposition.

\begin{prop}\label{prop:drifting-height}
Let $\gamma \subset N$ be an oriented, essential arc from cusp $C$ back to $C$,
which is disjoint from the interior of $C$. Let $g
\subset N$ be the geodesic in the homotopy class of $\gamma$. Let
$h(\gamma)$ be the height at which $\gamma$ enters $C$, and $h(g)$ be
the height at which $g$ enters $C$.

Assume that the orientation of $\gamma$ has been chosen to minimize
$\abs{h(g) - h(\gamma)}$. Then either $\abs{h(g) - h(\gamma)} \leq \sqrt{2}$, or 
\begin{equation}\label{eq:drifting-distance}
\abs{h(g) - h(\gamma)} \: \leq \: \frac{\ell(\gamma) - \ln \left(3 + 2 \sqrt{2} \right)  + 2 \sqrt{2}}{2}
\: = \: \frac{\ell(\gamma)}{2} + 0.5328...,
\end{equation}
where $\ell(\gamma)$ is the arclength of $\gamma$.
\end{prop}

\begin{proof}
Lift $\gamma$ to an arc $\widetilde{\gamma} \subset \HH^3$. The
oriented arc $\widetilde{\gamma}$ runs from horoball $H'$ to horoball
$H$. Let $\widetilde{g}$ be the corresponding lift of $g$, namely the
oriented geodesic from $H'$ to $H$.  Let $d_+ = d_+(g,\gamma)$ be the
distance along $\bdy H$ between the endpoints of $\widetilde{g}$ and
$\widetilde{\gamma}$ on $H$, and similarly let $d_- = d_-(g,\gamma)$
be the distance along $\bdy H'$ between the endpoints of
$\widetilde{g}$ and $\widetilde{\gamma}$ on $H'$. Since the
orientation of $\gamma$ has been chosen to minimize $\abs{h(g) -
  h(\gamma)}$, we have
\begin{equation}
\label{eq:h-gamma}
\abs{h(g) - h(\gamma)} \: \leq \: \min\{ d_-(g,\gamma), \, d_+(g,\gamma) \} 
                       \: \leq \: \frac{ d_- + d_+}{2}.
\end{equation}
Thus, to bound $\abs{h(g) - h(\gamma)}$, it will suffice to bound the
average of $d_-$ and $d_+$.

Next, we reduce the problem from three to two dimensions, as follows.
Consider cylindrical coordinates $(r, \theta, z)$ on $\HH^3$, with the
geodesic $\widetilde{g}$ at the core of the cylinder.  Thus $r$
measures distance from $\widetilde{g}$, while $\theta$ is the
rotational parameter, and $z$ measures distance along $\widetilde{g}$.
With these coordinates, the hyperbolic metric becomes
\begin{equation}\label{eq:cylindrical}
ds^2 \: = \: dr^2 + \sinh^2(r) \, d\theta^2 + \cosh^2(r) \, dz^2.
\end{equation}

We claim that no generality is lost by assuming $\widetilde{\gamma}$
lies in the half-plane corresponding to $\theta = 0$. This is because
the expression for the metric in \eqref{eq:cylindrical} is
diagonalized, hence the map $(r,\theta,z) \mapsto (r, 0,z)$ is
distance--decreasing. Thus replacing $\widetilde{\gamma}$ by its image
in this half-plane only makes it shorter. Furthermore, horoballs $H$
and $H'$ are rotationally symmetric about $\widetilde{g}$, hence the
new planar curve is still disjoint from their interiors. Finally,
observe that rotation about $\widetilde{g}$ keeps the endpoints of
$\widetilde{\gamma}$ a constant distance from $\widetilde{g} \cap H'$
and $\widetilde{g} \cap H$, respectively. Thus the quantities
$d_-(g,\gamma)$ and $d_+(g, \gamma)$ are unchanged when we replace
$\widetilde{\gamma}$ by its rotated image.  From now on, we will
assume that $\widetilde{\gamma}$ lies in the copy of $\HH^2$ for which
$\theta \in \{0, \pi\}$.

Recall that $\widetilde{\gamma}$ is disjoint from the interiors of $H$
and $H'$. When its endpoints are sufficiently far apart from $\widetilde{g}$, the geodesic between those points would pass through the interiors of the horoballs. Instead, the shortest path that stays outside $H$ and $H'$  follows the boundary of $H'$, then tracks a
hyperbolic geodesic tangent to $H'$ and $H$, then follows the boundary
of $H$. (See Figure~\ref{fig:drifting}.) The following lemma estimates
the length of the geodesic segment in the middle of
this path.

\begin{figure}
\begin{overpic}{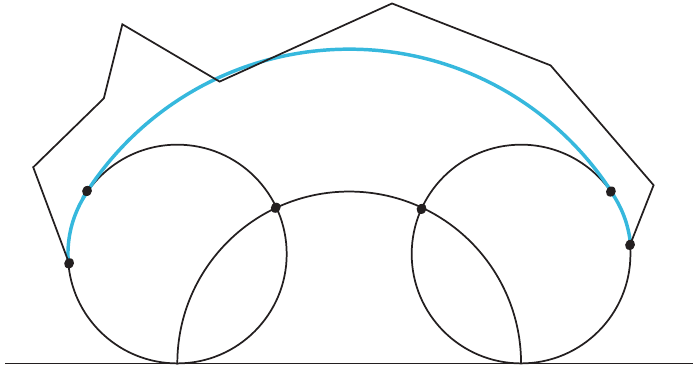}
\put(44,19.5){$\widetilde{g}  \subset \beta$}
\put(49,41){$\ell_1$}
\put(32,32){$\ell_2$}
\put(62,32){$\ell_2$}
\put(12,20){$\ell_0 - \ell_2$}
\put(73,20){$\ell_3 - \ell_2$}
\put(75,47){$\widetilde{\gamma}$}
\put(39,4){$H'$}
\put(87.5,4){$H$}
\end{overpic}
\caption{The setup of Proposition~\ref{prop:drifting-height}.  When the endpoints of $\widetilde{\gamma}$ are sufficiently far apart, the
shortest path that stays
disjoint from the interiors of $H$ and $H'$ is the three-piece blue
arc, of length $(\ell_0 - \ell_2) + \ell_1 + (\ell_3-\ell_2)$.}
\label{fig:drifting}
\end{figure}

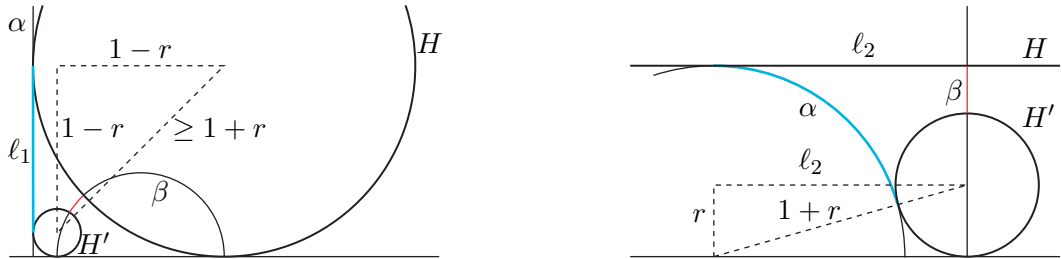
\begin{figure} [b]
\begin{overpic}{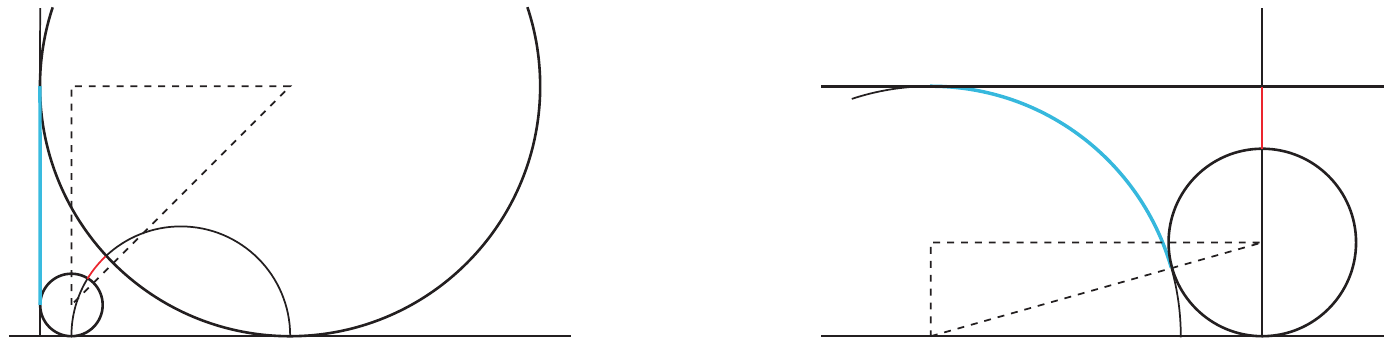}
\put(7,1){$H'$}
\put(39,20){$H$}
\put(0.5,22){$\alpha$}
\put(0.5,10){$\ell_1$}
\put(14,6){$\beta$}
\put(5.5,12){$1-r$}
\put(16,12){$\geq1+r$}
\put(10,19.5){$1-r$}
\put(96,19.5){$H$}
\put(96,13){$H'$}
\put(80,20){$\ell_2$}
\put(75,14){$\alpha$}
\put(88.8,15.5){$\beta$}
\put(73,4.3){$1+r$}
\put(65,4){$r$}
\put(75,9){$\ell_2$}
\end{overpic}
\caption{The setup of Lemma~\ref{lemma:tangent-length}. Left: making
geodesic $\alpha$ vertical helps estimate $\ell_1$. Right: making
geodesic $\beta$ vertical helps estimate $\ell_2$. In both panels, the
horoball $H'$ has Euclidean radius $r$.}
\label{fig:pythagorean}
\end{figure}

\begin{lemma}
\label{lemma:tangent-length}
Let $H$ and $H'$ be horoballs in $\HH^2$ with disjoint interiors. Let
$\alpha \subset \HH^2$ be a hyperbolic geodesic such that $H$ and $H'$
are both tangent to $\alpha$, on the same side of $\alpha$. Let
$\beta\subset \HH^2$ be a hyperbolic geodesic perpendicular to both
$H$ and $H'$. If $\ell_1$ denotes the length along $\alpha$ from $H
\cap \alpha$ to $H' \cap \alpha$, and $\ell_2$ denotes the length
along $\bdy H$ from $H \cap \alpha$ to $H \cap \beta$, then
\[
\ell_1 \: \geq \: \ln \left( 3 + 2 \sqrt{2} \right) \qquad \mbox{and}
\qquad \ell_2 \: \leq \: \sqrt{2}.
\]
Each inequality becomes equal if and only if $H$ is tangent to $H'$.
\end{lemma}


\begin{proof}
Let $g$ be the geodesic segment of $\alpha$ whose length is $\ell_1$.
For the first inequality, apply an isometry of $\HH^2$ so that
$g \subset \alpha$ is vertical in the upper half-plane model, so that $H$ is the
larger horoball, and so that the Euclidean radius of $H$ is $1$. Then
the point of tangency $\alpha \cap H$ is at Euclidean height $1$.
(See Figure~\ref{fig:pythagorean}, left.) A calculation with the
Pythagorean theorem then implies that the Euclidean radius of $H'$
must be $r \leq 1/(3 + 2 \sqrt{2} )$, with equality if and only if $H$
is tangent to $H'$. Since one endpoint of $g$ is at height $1$ and the
other endpoint is at height $r$, we have $\ell_1 = \ln(1/r) \geq \ln(3
+ 2 \sqrt{2} )$.

For the second inequality, apply an isometry of $\HH^2$ so that $\bdy
H$ is a horizontal line at Euclidean height $1$. Then $\alpha$ is a
Euclidean semicircle of radius $1$, and the horoball $H'$ must have
Euclidean radius $r \leq 1/2$. (See Figure~\ref{fig:pythagorean},
right.) Again, a calculation with the Pythagorean theorem implies that
$\ell_2 \leq \sqrt{2}$, with equality if and only if $r= 1/2$.
\end{proof}

Returning to the proof of Proposition~\ref{prop:drifting-height}, we
import the notation of Lemma~\ref{lemma:tangent-length}. That is: the
geodesic $\alpha$ is tangent to $H$ and $H'$, while the geodesic
$\beta$ contains $\widetilde{g}$. Let $\ell_1$ and $\ell_2$ be as in
the lemma. Let $\ell_0 = d_-(g,\gamma)$ be the distance along $\bdy
H'$ between the $\widetilde{g}\cap \bdy H'$ and $\widetilde{\gamma}
\cap \bdy H'$, and let $\ell_3 = d_+(g,\gamma)$.

If one of $\ell_0 = d_-$ or $\ell_3 = d_+$ is no longer than
$\sqrt{2}$, then equation \eqref{eq:h-gamma} gives $\abs{h(g) -
  h(\gamma)} \leq \sqrt{2}$ as well, and the proof is
complete. Otherwise, if $\ell_0$ and $\ell_3$ are both longer than
$\sqrt{2}$, then Lemma~\ref{lemma:tangent-length} implies that they
are longer than $\ell_2$. As a consequence, the shortest path between
the endpoints of $\widetilde{\gamma}$ that stays disjoint from $H$ and
$H'$ will need to track horoball $H'$ for distance $(\ell_0 -
\ell_2)$, then follow geodesic $\alpha$ for distance $\ell_1$, then
track horoball $H$ for distance $(\ell_3 - \ell_2)$. See
Figure~\ref{fig:drifting}.

 Thus we may compute:
$$
\begin{array}{r c l c l}
\ell(\gamma) & \geq & (\ell_0 - \ell_2) + \ell_1 + (\ell_3 - \ell_2) & & \mbox{by construction of $\widetilde{\gamma}$} \\
& = & (d_- + d_+)  - 2 \ell_2 + \ell_1 & & \mbox{by the definition of $\ell_0$ and $\ell_3$} \\
& \geq & (d_- + d_+)  - 2 \sqrt{2} +  \ln (3 + 2 \sqrt{2}) & & \mbox{by Lemma~\ref{lemma:tangent-length},} \\
& \geq & 2  \abs{h(g) - h(\gamma)}  - 2 \sqrt{2} +  \ln (3 + 2 \sqrt{2}) & & \mbox{by \eqref{eq:h-gamma},} \end{array}
$$
implying \eqref{eq:drifting-distance}.
\end{proof}

We can now prove the upper bounds of Theorem~\ref{thm:main-qf}.

\begin{proof}[Proof of Theorem~\ref{thm:main-qf}, upper bound]
Let $C \subset N$ be a maximal cusp corresponding to the puncture $p$
of $F$. As above, $\bdy C \cap \core(N)$ decomposes into three compact
annular bands: the band $B_-$ between heights $h_-(N)$ and $h(a_0)$,
the band $B$ between heights $h(a_0)$ and $h(a_k)$, and the band $B_+$
between heights $h(a_k)$ and $h_+(N)$.

The area and height of $B$ were bounded in Lemma
\ref{lemma:upper-qf-start}. As for $B_-$, let $a_0 \in \Delta_-(N)$ be
one of the arcs from $C$ to $C$ that is shortest on $\bdy_-
\core(N)$. Let $\gamma$ be the geodesic in $\bdy_-\core(N)$ in the
homotopy class of $a_0$. By Lemmas~\ref{lemma:length-bound} and
\ref{lemma:pleated-cusp-size},
\begin{equation}\label{eq:gamma-length}
\ell(\gamma) \: \leq \: 2 \ln \abs{ 6 \chi (F) \, / \, 2 ^{1/4 }} \: =
\: 2 \ln \abs{ \chi(F) } + \ln  \left( 18 \sqrt{2}  \right).
\end{equation}
Note that $\ln  \left( 18 \sqrt{2}  \right) \approx 3.2369 > \sqrt{2}$, hence the larger upper bound in Proposition~\ref{prop:drifting-height} is the one in equation \eqref{eq:drifting-distance}.
Thus, by plugging estimate \eqref{eq:gamma-length} into \eqref{eq:drifting-distance}, we obtain
\[
\height(B_-) \: = \: \abs{ h(a_0) - h(\gamma) } \: \leq \: \ln \abs{
  \chi(F) } + \frac{  \ln  \left( 18 \, \sqrt{2} \right)  }{2} + 0.54
\: < \: \ln \abs{ \chi(F) } + 2.16.
\]
By Lemma~\ref{lemma:horocycle-length}, the circumference of $B_-$
(which is a longitude of the cusp $C$) satisfies $\lambda \leq - 6
\chi(F)$. Thus
\[
\area(B_-) \: = \: \lambda \cdot \height(B_-) 
  \: < \: \big| 6 \chi(F) \ln \abs{ \chi(F) } + 13 \chi(F) \big|.
\]
The top band $B_+ = B_+(N)$ satisfies the same estimates. 

Combining these estimates with Lemma~\ref{lemma:upper-qf-start}, we
obtain
\[
\area(\bdy C \cap \core(N)) \: = \:  \area(B_- \cup B \cup B_+)  \: <
\: 9 \chi(F)^2 \, d_\arc(N,p) + \Big| 12 \chi(F)   \ln \abs{ \chi(F) }
+ 26 \chi(F) \Big|.
\]
Similarly,
\[
\height(\bdy C \cap \core(N)) \: = \:  \height(B_- \cup B  \cup B_+)
\: < \:  -3 \, \chi(F) \, d_\arc(N, p) + 2  \ln \abs{ \chi(F) } + 5,
\]
completing the proof.
\end{proof}


\section{Sweepouts}
\label{sec:sweepouts}

In this section, we describe an important geometric and topological
construction needed for the lower bounds in Theorems~\ref{thm:main}
and~\ref{thm:main-qf}.

\begin{define}
\label{def:sweepout}
Let $N$ be a hyperbolic $3$--manifold, $F$ a surface, and $f_0 \from F
\to N$ a map sending punctures to cusps.  Fix a connected set $J
\subset \RR$.  A {\em sweepout through $f_0$} is a map $\Psi \from F
\times J \to N$, thought of as a one-parameter family of maps $\Psi_t
\from F \to N$, each homotopic to $f_0$.

A sweepout $\Psi$ is called \emph{geometric} if every $\Psi_t$ is a
\emph{simplicial hyperbolic map}: that is, for every $t \in J$, the
image $F_t = \Psi_t(F)$ is a hyperbolic cone surface with at most one
cone point of angle $2\pi \leq \theta_t < 4\pi$. Note that a pleating
map along an ideal triangulation is a special case of a simplicial
hyperbolic map.

Let $g_t$ be the hyperbolic cone metric on $F$ induced by $\Psi_t$.
Then the continuity of $\Psi$ implies that $g_t$ varies continuously
with $t$. In particular, the lengths of geodesic realizations of
homotopy classes of arcs and curves (with respect to $g_t$) vary
continuously with $t$.
\end{define}

\begin{define}
\label{def:equivariant-sweep}
Let $M = M_\psi$ be a fibered hyperbolic $3$--manifold with fiber $F$
and monodromy $\psi$.  Let $N = N_\psi$ be the infinite cyclic cover
of $M$, with deck transformation $Z \from N \to N$.  Fix $r > 0$ and
define $z \from F \cross \RR \to F \cross \RR$ by $z(x, t) = (\psi(x),
t + r)$.  We say a sweepout $\Psi \from F \times \RR \to N$ is {\em
  equivariant} if each $\Psi_t$ is properly homotopic to the fiber,
and
\[
Z \circ \Psi = \Psi \circ z. 
\]
\end{define}

Note that equivariance implies that $\Psi$ descends to a degree-one
sweepout of $M$. 

\begin{prop}
\label{prop:sweepout}
Let $M = M_\psi$ be a fibered hyperbolic $3$--manifold with fiber $F$
and monodromy $\psi$.  Let $N = N_\psi$ be the infinite cyclic cover.
Then there is a geometric, equivariant sweepout $\Psi \from F \times
\RR \to N$.
\end{prop}

This result originates in the work of Thurston~\cite[Theorem
9.5.13]{thurston:notes}; see page 9.47 in particular. A careful
account of the proof was also written down by Canary~\cite[Sections
4--5]{canary:covering}. What follows below is a review of their
argument, adapted to ideal triangulations.

\begin{proof}[Proof of Proposition~\ref{prop:sweepout}]
Let $\tau_0$ be an ideal triangulation of $F$.  Hatcher
proved~\cite{hatcher:triangulations} that the triangulations $\tau_0$
and $\psi(\tau_0)$ can be connected by a sequence of diagonal
exchanges, as in Figure~\ref{fig:diagonal-exchange}.  Thus we have a
sequence of ideal triangulations, $\tau_0, \tau_1, \ldots, \tau_r =
\psi(\tau_0)$, where each $\tau_i$ differs from $\tau_{i-1}$ by a
diagonal exchange.  We extend the sequence of triangulations $\tau_i$
to a bi-infinite sequence $\{\tau_i \st i \in \ZZ \}$, such that
$\tau_{i+r} = \psi(\tau_i)$.

\begin{figure}[h]
\begin{overpic}{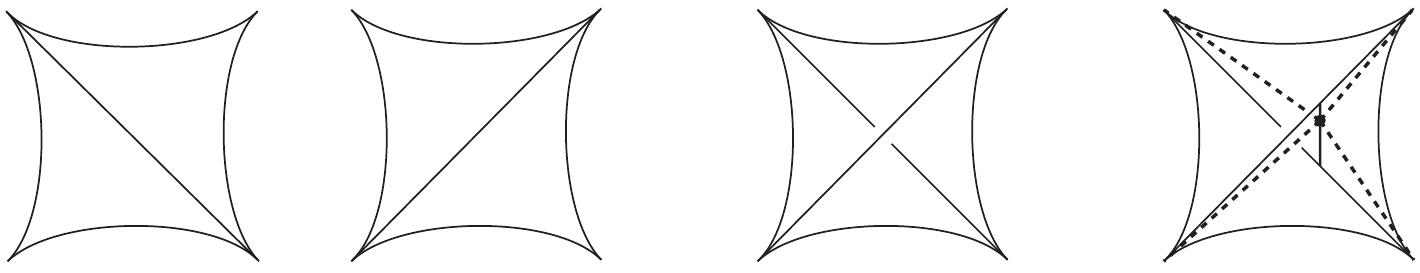}
\put(9.5,10){$\eps_{i-1}$}
\put(31,10){$e_i$}
\put(21,9){$\Rightarrow$}
\put(94,10){$x_t$}
\end{overpic}
\caption{Left: a diagonal exchange in a quadrilateral. Center right: a
diagonal exchange between two pleated surfaces creates a
$3$--dimensional tetrahedron $\Delta_i$. Right: a singular
quadrilateral $Q_t$ interpolates between the top and bottom pleated
sides of the tetrahedron.}
\label{fig:diagonal-exchange}
\end{figure}
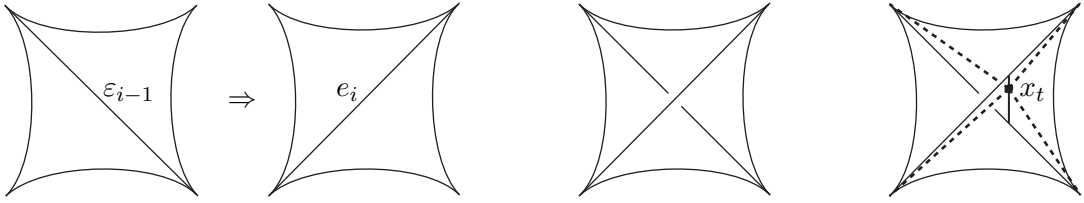

Fix an embedding $f_0 \from F \to N$ isotopic to the fiber.  By
Proposition~\ref{prop:pleating}, for each $i \in \ZZ$ we may take
$\Psi_i \from F \to N$ to be a pleating of $f_0$ along $\tau_i$.  We
choose these pleating maps so that $Z \circ \Psi_i = \Psi_{i+r} \circ
\psi$.  Define $F_i = \Image(\Psi_i)$ and notice that, since $\tau_i$
differs from $\tau_{i-1}$ by a diagonal exchange, $F_i$ differs from
$F_{i-1}$ by an ideal tetrahedron $\Delta_i$.

Fix $i \in \ZZ$.  Let $\eps_{i-1}$ be the edge of $\tau_{i-1}$ that is
exchanged for the edge $e_i$ of $\tau_i$. Because all six edges of the
tetrahedron $\Delta_i$ lift to hyperbolic geodesics, the edges
$\eps_{i-1}$ and $e_i$ lift to hyperbolic geodesics with no shared
endpoints at infinity. In $\HH^3$, this pair of geodesics is joined by
a unique geodesic segment $\gamma$ that meets $\eps_{i-1}, e_i$
perpendicularly. (In the special case where $\eps_{i-1}$ and $e_i$
intersect, $\gamma$ has length $0$.)

Now, for every $t \in [i-1,i]$, let $x_t$ be the point on
$\gamma$ that is distance $\ell(\gamma)(t-i+1)$ from $\eps_{i-1}$ and
distance $\ell(\gamma)(i-t)$ from $e_i$. In other words, $x_t$ is
obtained by linear interpolation between the points where $\gamma$
meets $\eps_{i-1}$ and $e_i$. We construct a (singular) ideal
quadrilateral $Q_t$ by coning $x_t$ to the four edges of $\Delta_i
\setminus ( e_i \cup \eps_{i-1} ) $. (See Figure
\ref{fig:diagonal-exchange}, right.) Finally, let $F_t$ be the surface
that includes the quadrilateral $Q_t$ inside the tetrahedron
$\Delta_i$, and agrees with $F_i$ everywhere else. Recall that
$F_{i-1}$ agrees with $F_i$ outside $\Delta_i$.

For $t \in (i-1,i)$, we take $\Psi_t \from F \to N_\psi$ to be a
simplicial hyperbolic map with image $F_t$.  We choose the maps
$\Psi_t$ so that $\Psi|_{F \cross [i-1,i]}$ is continuous and so that
$Z \circ \Psi_t = \Psi_{t+r} \circ \psi$.

Now consider the geometry of $F_t = \Psi_t(F)$. Wherever this surface
agrees with $F_i$, it is built out of ideal triangles, and inherits an
intrinsically hyperbolic metric from $N_\psi$. Meanwhile, the
quadrilateral $Q_t$ where $F_t$ disagrees with $F_i$ is constructed
out of four $(2/3)$--ideal triangles that share a vertex at
$x_t$. Thus the surface $F_t$ has a smooth hyperbolic metric
everywhere except at $x_t$. At this cone point, observe that the
singular quadrilateral $Q_t$ is not contained in any hyperbolic
half-space through $x_t$. As a result, a lemma of Canary~\cite[Lemma
4.2]{canary:covering} implies that the cone angle at $x_t$ is
$\theta_t \geq 2\pi$. Also, because each of the four triangles meeting
at $x_t$ has an interior angle less than $\pi$, we have $\theta_t <
4\pi$.

Since the triangulations $\tau_i$ satisfy $\tau_{i+r} =
\psi(\tau_i)$, we have
\[
Z \circ \Psi_{t} = \Psi_{t+r} \circ \psi, 
    \quad \mbox{for all $t \in \RR$.} 
\]
Therefore, $\Psi$ is the desired equivariant, geometric sweepout of
$N$. 
\end{proof}

The above construction extends nicely to quasi-Fuchsian manifolds.

\begin{prop}\label{prop:qf-sweepout}
Let $N \cong F \times \RR$ be a cusped quasi-Fuchsian
$3$--manifold. Let $\tau, \tau'$ be ideal triangulations of $F$. Then
there exists a geometric sweepout $\Psi \from F \times [0, r] \to N,$
such that $\Psi_0$ is the pleating map along $\tau$ and $\Psi_r$ is
the pleating map along $\tau'$.
\end{prop}

\begin{proof}
We repeat the proof of Proposition~\ref{prop:sweepout}, without
needing to worry about equivariance. Let $\tau = \tau_0, \tau_1,
\ldots, \tau_r = \tau'$ be a sequence of ideal triangulations of $F$,
with each $\tau_i$ differing from $\tau_{i-1}$ by a diagonal
exchange. Then each $\tau_i$ can be realized by a pleated surface
$F_i$, and we may interpolate from $F_{i-1}$ to $F_i$ by a
$1$--parameter family of simplicial hyperbolic surfaces, as in Figure
\ref{fig:diagonal-exchange}.
\end{proof}

\begin{lemma}
\label{lemma:cuff}
Let $\Psi \from F \times [0, r] \to N$ be a geometric sweepout in a
hyperbolic $3$--manifold $N$.  Let $C$ be an embedded horocusp in $N$,
with longitude of length $\lambda$. Then, for every surface $F_t =
\Psi_t(F)$ in the sweepout, $C \cap F_t$ contains an equidistant cusp
neighborhood whose area is at least $\lambda$.
\end{lemma}    

Recall that an \emph{equidistant cusp} of a hyperbolic cone surface is
defined in Definition~\ref{def:equidistant-cusp}.

\begin{proof}[Proof of Lemma~\ref{lemma:cuff}]
This is identical to the proof of Lemma~\ref{lemma:pleated-cusp}, with
pleated surfaces replaced by simplicial hyperbolic surfaces. First,
take a horocusp $C_0 \subset C$ small enough so that $F_t \cap C_0$ is
a non-singular, horospherical cusp $R_t^0$. If $d$ denotes the
distance between $\bdy C_0$ and $\bdy C$, then $\area(R_t^0) \geq
e^{-d} \lambda$. Then $F_t \cap C$ contains an equidistant
$d$--neighborhood of $R_t^0$. By Lemma~\ref{lemma:expansion-rate},
this equidistant neighborhood $R_t$ satisfies
\[
\area(R_t) \: \geq \: e^d \, \area(R_t^0) \: \geq \: \lambda.
\qedhere
\]
\end{proof}


\section{Lower bound: fibered manifolds}
\label{sec:lower}

In this section, we prove the lower bound of Theorem~\ref{thm:main}.
We begin with a slightly more restricted statement:

\begin{theorem}
\label{thm:lower}
Let $F$ be an orientable hyperbolic surface with a preferred puncture
$p$, and let $\psi \from F \to F$ be an orientation-preserving,
pseudo-Anosov homeomorphism such that $\psi(p)=p$. In the mapping
torus $M_\psi$, let $C$ be a horocusp corresponding to $p$, whose
longitude has length $\lambda = 2^{1/4}$. Then there exists an integer
$n \geq 1$, such that
\[
\area(\bdy C) \: > \:   \frac{ d_{\arc} (\psi^n)}{450\,  \chi(F)^4 }
\quad \mbox{and} \quad
\height(\bdy C) \: > \: \frac{ d_{\arc} (\psi^n)}{536 \,
  \chi(F)^4 } .
\]
\end{theorem}

Theorem~\ref{thm:lower} differs from the lower bound of Theorem
\ref{thm:main} in several small ways. First, it restricts attention to
horocusps that have longitude of length $2^{1/4}$. (This choice of
longitude is justified by Lemma~\ref{lemma:waist}, which represents
the best available lower bound on the longitude.)  Second, Theorem
\ref{thm:lower} restricts attention to monodromies that fix the
puncture $p$.  (Given an arbitrary pseudo-Anosov $\mon$, one can let
$\psi$ be the smallest power of $\mon$ such that $\psi(p)=p$.)
Finally, Theorem~\ref{thm:lower} estimates cusp area and height in
terms of the translation distance $d_\arc(\psi^n)$ for an unspecified
integer $n$, rather than the stable translation distance
$\dbar_\arc(\psi)$. We shall see at the end of the section that this
restricted statement quickly implies the lower bound of Theorem
\ref{thm:main}.

\begin{proof}[Proof of Theorem~\ref{thm:lower}] 
As in Sections~\ref{sec:upper} and~\ref{sec:sweepouts}, it is
convenient to work with the infinite cyclic cover of $M_\psi$, namely
$N_{\psi} \cong F \times \RR$. The horocusps of $M$ lift to annular,
rank one cusps.  Let $A \subset N_{\psi}$ be the lift of $\bdy C$ that
corresponds to the puncture $p$ of $F$.  Then $A$ is an annulus, with
longitude of length $\lambda = 2^{1/4}$, which covers the torus $\bdy
C$.

By Proposition~\ref{prop:sweepout}, there is a geometric, equivariant
sweepout $\Psi \from F \times \RR \to N_\psi$. In particular, for
every $t$, we have $F_t = \Psi_t(F)$ is a hyperbolic cone surface with
at most one singular point of cone angle $2\pi \leq \theta_t < 4\pi$.

\begin{define}
\label{def:short-arc}
Let $F_t$ be a simplicial hyperbolic surface in $N$ (such as one
occurring in the sweepout). We say that an arc $a \in \arc^{(0)}(F,p)$
is \emph{short} on $F_t$ if $a$ runs from $p$ to $p$, and if its
geodesic representative in the hyperbolic metric of $F_t$ is shortest
among all such arcs.

Note that a given surface $F_t$ can have multiple short arcs. However,
all short arcs on $F_t$ have disjoint representatives, by
Lemma~\ref{lemma:embeddedness}.
\end{define}

For each surface $F_t$, Lemma~\ref{lemma:length-bound} gives an
explicit upper bound on the length of a short arc in $F_t$ (hence,
also on its length in $N_\psi$). On the cusp annulus $A$, each short
arc gives a shadow of a horoball, with area bounded below. More
precisely, we obtain the following quantitative estimate.

\begin{lemma}
\label{lemma:horoball-shadows}
Let $a \in \arc^{(0)}(F,p)$ be an arc that is short on $F_t$ for some
$t$.  Then, on the cusp annulus $A \subset N_\psi$, the arc $a$
corresponds to a pair of disjoint disks, each of radius
\[
r = \frac{\sqrt{2}}{8 \pi^2 \, \chi(F)^2} \, .
\]
Furthermore, if an arc $b$ is short on $F_{t'}$ for some $t'$, and $a
\neq b \in \arc^{(0)}(F,p)$, then the disks corresponding to $a$ and
$b$ are disjoint on $A$.
\end{lemma}

\begin{proof}
By convention (and by Lemma~\ref{lemma:waist}), the longitude of $A$
has length $\lambda = 2^{1/4}$. Thus, by Lemma~\ref{lemma:cuff}, the
intersection between $F_t$ and the horocusp of $N_\psi$ contains an
equidistant cusp neighborhood $R_t$, of area at least $2^{1/4}$.
Therefore, by Lemma~\ref{lemma:length-bound}, the length of a short
arc $a$ in $F_t \setminus R_t$ is
\[
\ell(a) < 2 \ln \abs{ 2\pi \, \chi(F) / 2^{1/4} }.
\]
Since $F_t$ is immersed in $N_\psi$ as a piecewise geodesic union of
hyperbolic triangles, this immersion is distance-decreasing. Thus, in
$N_\psi$, the geodesic $g_a$ in the homotopy class of $a$ must also be
shorter than the above estimate.

Lift $N_\psi$ to the universal cover $\HH^3$ in the upper half-space
model, so that the cusp annulus $A$ lifts to a horizontal horosphere
at Euclidean height $1$. Let $H_\infty$ be the horoball above this
horosphere. This means that $g_a$ lifts to a vertical geodesic that
starts at height $1$ and ends at the top of a horoball $H_a$, of
diameter
\begin{equation}
\label{eq:horoball-height}
e^{-\ell(a)} \geq \frac{\sqrt{2}}{4 \pi^2 \, \chi(F)^2} \, .
\end{equation}

In $\HH^3$, there is a covering transformation for $N_\psi$ that maps
$H_a$ to $H_\infty$, and maps $H_\infty$ to another horoball
$H'_a$. The diameter of $H'_a$ must be the same as that of $H_a$
because they lie at the same distance from $H_\infty$ (namely, the
length of $g_a$). This new horoball $H'_a$ cannot belong to the same
orbit as $H_a$ under the parabolic subgroup $\pi_1(A) = \ZZ$
preserving $H_\infty$: for, this parabolic subgroup preserves the
orientation on all the lifts of $g_a$, but one lift of $g_a$ is
oriented downward toward $H_a$ whereas the other is oriented upward
from $H'_a$ toward $H_\infty$. Thus the shadows of $H_a$ and $H'_a$
are disjoint disks on the horosphere at height $1$, which project to
disjoint disks $D_a$ and $D'_a$ on $A$ because $H_a$ and $H'_a$ are in
different orbits. After shrinking $D_a$ and $D'_a$ if necessary, we
obtain a pair of disjoint disks of radius $\sqrt{2} / 8 \pi^2 \,
\chi(F)^2$.

Now, suppose that $b$ is another arc from $p$ to $p$, not isotopic to
$a$, such that $b$ is short on $F_{t'}$ for some $t'$. Then,
performing the same construction as for $a$, we obtain a pair of
horoballs $H_b$ and $H'_b$, whose heights also satisfy equation
\eqref{eq:horoball-height}. The four horoballs $H_a, H'_a, H_b, H'_b$
must lie in distinct orbits of the parabolic $\ZZ$ subgroup preserving
$\infty$, because $a$ and $b$ are in distinct homotopy classes on
$F$. Thus the four horoballs are disjoint in $\HH^3$. Their shadows on
$H_\infty$ are not necessarily disjoint. However, if we shrink all
four horoballs until their diameter is exactly $\sqrt{2} / 4\pi^2 \,
\chi(F)^2$, then the shadows of disjoint horoballs of the same size are
themselves disjoint.

Since $H_a, H'_a, H_b, H'_b$ lie in distinct orbits under $\pi_1(A)$,
we obtain four disjoint disks $D_a, D'_a, D_b, D'_b$ in $A$, each of
radius $\sqrt{2} / 8 \pi^2 \, \chi(F)^2$.
\end{proof}

To obtain a lower bound on the area and height of $\bdy C$, we need to
find a sequence of arcs in $F$, each of which is short in a surface
$F_t$ for some $t$.

\begin{lemma}
\label{lemma:arc-sequence}
There is a sequence of real numbers $0 = t_0, t_1, \ldots, t_k = r$
and an associated sequence $a_0, a_1, \ldots, a_k$ of arcs embedded in
$F$, with the following properties:
\begin{enumerate}
\item 
Each $a_i$ is short on $F_{t_i}$.
\item 
The first and last arcs satisfy $a_k = \psi(a_0)$.
\item 
Each $a_i$ is disjoint from $a_{i-1}$. In other words, $[a_{i-1},
a_i]$ is an edge of $\arc(F,p)$.
\end{enumerate}
\end{lemma}

\begin{proof}
Let $a$ be an embedded arc in $F$ from $p$ to $p$. Define
\[
S(a) := \{ t \in \RR 
  \st \mbox{$a$ is shortest on $F_t$ among all arcs from $p$ to $p$} \}.
\]
In other words, $S(a)$ consists of those values of $t$ for which $a$
is short (as in Definition~\ref{def:short-arc}).  According to the
definition of a geometric sweepout (Definition~\ref{def:sweepout}) the
length of any arc varies continuously with $t$.  Since being shortest
is a closed condition it follows that the set $S(a)$ is closed.  Also,
since every surface $F_t$ in the sweepout has a short arc, the line
$\RR$ is covered by the sets $S(a)$, as $a$ varies over the vertices
of $\arc(F,p)$.

We claim that the arcs $a$ for which $S(a) \neq \emptyset$ belong to
finitely many $\psi$--orbits.  This is because every $\psi$--orbit of
arcs in $F$ descends to a single arc in $M_\psi$, with distinct orbits
descending to distinct arcs. The two endpoints of a geodesic $g_a
\subset M_\psi$ representing the arc $a$ must be distinct in $\bdy C$
(otherwise, a deck transformation of $M_\psi$ would reverse the
orientation of a lift of $g_a$, fixing a point in the middle).  The
two endpoints of $g_a$ on $\bdy C$ are the centers of disjoint disks
guaranteed by Lemma~\ref{lemma:horoball-shadows}.  Thus, by
Lemma~\ref{lemma:horoball-shadows}, every orbit of arcs that is
shortest on some $F_t$ makes a definite contribution to $\area(\bdy
C)$. On the other hand, $\area(\bdy C)$ is bounded above (e.g.\ by
Theorem~\ref{thm:upper}), hence there can be only finitely many
$\psi$--orbits of arcs for which $S(a) \neq \emptyset$.

Next, we claim that each $S(a)$ is compact. Fix an arc $a$, and let
$g_a$ be the geodesic in $N_\psi$ that is homotopic to $a$. Note that
the sweepout of $N_\psi$ must eventually {\em exit}: that is, as
$\abs{t} \to \infty$, $F_t$ must leave any compact set in
$N_\psi$. Thus, as $\abs{t} \to \infty$, the distance from $F_t$ to
the geodesic $g_a$ becomes unbounded (outside the horocusp
$C$). However, any path in $F_t$ that is homotopic to $g_a$ but
remains outside an $s$--neighborhood of $g_a$ must be extremely long
(with length growing exponentially in $s$). Therefore, when $\abs{t}
\gg 0$, the geodesic representative of $a$ on $F_t$ must be very long,
and in particular cannot be the shortest arc on $F_t$. This means that
$S(a)$ is bounded, hence compact.

We can now conclude that there are only finitely many arcs $a$ for
which $S(a) \cap [0,r] \neq \emptyset$. By the first claim above,
these arcs belong to finitely many $\psi$--orbits. Within each orbit,
the compact sets $S(a)$ and $S(\psi(a))$ differ by a translation by
$r$. Thus only finitely many sets in each $\psi$--orbit can intersect
$[0,r]$.

Let $a_0$ be an arc that is short on $F_0$; that is, $0 \in
S(a_0)$. Then $\psi(a_0)$ is short on $F_r$. Since the connected
interval $[0,r]$ is covered by finitely many closed sets $S(a)$, a
lemma in point-set topology (Lemma~\ref{lemma:appendix} in the
Appendix) implies that one may ``walk'' from $S(a_0)$ to
$S(\psi(a_0))$ by intersecting sets: there is a real number $t_1 \in
S(a_0) \cap S(a_1)$, a number $t_2 \in S(a_1) \cap S(a_2)$, and so on,
for arcs $a_0, \ldots, a_k = \psi(a_0)$.

By definition, $t_i \in S(a_{i-1}) \cap S(a_i)$ means that both
$a_{i-1}$ and $a_i$ are short on $F_{t_i}$. Thus, by Lemma
\ref{lemma:embeddedness}, $[a_{i-1}, a_i]$ is an edge of
$\arc(F,p)$. We have therefore constructed a walk through the
$1$--skeleton of $\arc(F,p)$, from $a_0$ to $a_k = \psi(a_0)$, through
arcs $a_i$ that are each short on some simplicial hyperbolic surface.
\end{proof}

\begin{lemma}
\label{lemma:power-sequence}
The sequence of arcs $a_0, a_1, \ldots, a_k$ in Lemma
\ref{lemma:arc-sequence} contains a subsequence 
$b_0, b_1,
\ldots, b_m$ with the following properties:
\begin{enumerate}
\item 
Each $b_i$ is short on $F_{t_i}$. 
\item 
The arcs $b_1, \ldots, b_{m}$ are all in distinct $\psi$--orbits.
\item 
The first and last arcs satisfy $b_m = \psi^n(b_0)$, for some integer
$n \neq 0$.
\item 
Each $b_i$ is disjoint from $b_{i-1}$. In other words, $[b_{i-1},
b_i]$ is an edge of $\arc(F,p)$.
\end{enumerate}
\end{lemma}

\begin{proof}
The arcs $a_0, a_1, \ldots, a_k$ in Lemma~\ref{lemma:arc-sequence}
constitute a walk through the $1$--skeleton of $\arc(F,p)$, from $a_0$
to $a_k = \psi(a_0)$. Given that this walk exists, one can excise some
of the $a_i$ if necessary to form a loop-erased walk $a_0$ to $a_k =
\psi(a_0)$. That is, one may walk from $a_0$ to $a_k = \psi(a_0)$
through some subcollection of the $a_i$, without visiting the same
isotopy class more than once.

Next, suppose that there are indices $i<j$, such that $a_i$ and $a_j$
belong to the same $\psi$--orbit. Without loss of generality, assume
that $i,j$ are an innermost pair with this property. This means that
$a_j = \psi^n(a_i)$ for some $n \neq 0$, and $a_{i+1}, \ldots, a_j$
are all in distinct $\psi$--orbits. Now, we simply restrict attention
to the subsequence from $i$ to $j$.  That is, let $b_0 = a_i$, $b_1 =
a_{i+1}$, and so on, until $b_m = a_j$ for $m = j-i$. This subsequence
satisfies the lemma.
\end{proof}

We can now complete the proof of Theorem~\ref{thm:lower}.  Notice that
in Lemma~\ref{lemma:power-sequence}, $b_0, b_1, \ldots, b_m$ are the
vertices of a path through $\arc^{(1)}(F,p)$ from $b_0$ to
$\psi^n(b_0)$. Thus $m \geq d_{\arc} (\psi^n)$. By Lemma
\ref{lemma:horoball-shadows}, each arc $b_i$ corresponds to two
disjoint disks $D_i, D'_i \subset A$, each of radius $\sqrt{2} / 8
\pi^2 \, \chi(F)^2$. Furthermore, for $j \neq i$, the disks $D_i,
D'_i, D_j, D'_j$ are all disjoint. Thus we have at least $2d_{\arc}
(\psi^n)$ disjoint disks altogether. Since the arcs $b_1, \ldots, b_m$
are all in different $\psi$--orbits, these disks project to $2d_{\arc}
(\psi^n)$ disjoint disks on the cusp torus $\bdy C \subset M_\psi$.

To obtain a lower bound on $\area(\bdy C)$, we sum up the areas of
these disjoint disks, and multiply by the circle packing constant of
$2\sqrt{3}/\pi$ (see~\cite[Theorem 1]{boroczky}). Thus
\[
\area(\bdy C) \: \geq \: 
\frac{2 \sqrt{3}}{\pi} \cdot 2d_{\arc} (\psi^n) \cdot \pi \left(  \frac{\sqrt{2}}{ 8 \pi^2 \, \chi(F)^2} \right)^{\!\! 2}
\: = \:  \frac{ \sqrt{3} \, d_{\arc} (\psi^n) }{8 \pi^4 \, \chi(F)^4} 
\: > \: \frac{ d_{\arc} (\psi^n) }{450 \, \chi(F)^4} \, .
\]
Finally, since $\area(\bdy C) = \lambda \cdot \height(\bdy C)$, and
we have normalized the horocusp so that $\lambda = 2^{1/4}$, we have
\[
\height(\bdy C)
\: \geq \:  \frac{ \sqrt{3} \, d_{\arc} (\psi^n) }{2^{1/4} \, 8 \pi^4 \, \chi(F)^4} 
\: > \: \frac{ d_{\arc} (\psi^n) }{536 \, \chi(F)^4} \, .
\qedhere
\]
\end{proof}

\begin{proof}[Proof of Theorem~\ref{thm:main}, lower bound]
Let $\mon \from F \to F$ be a pseudo-Anosov homeomorphism, and let
$\psi = \mon^n$ be the smallest power of $\mon$ that fixes the
puncture $p$. Let $C$ be an embedded cusp of $M_\mon$ corresponding to
$p$, whose longitude has length $\lambda = 2^{1/4}$. By Lemma
\ref{lemma:waist}, an embedded cusp of this size exists, and is
smaller than the maximal cusp of $M_\mon$. Thus all lower bounds on
$C$ also apply to the maximal cusp.

In the cover $M_\psi$, $C$ lifts to an embedded horocusp $C_1 \subset
M_\psi$, which is a one-sheeted cover of $C$. Furthermore, for every
integer $m \geq 1$, the mapping torus $M_{\psi^m}$ contains an
embedded horocusp $C_m$ whose longitude has length $2^{1/4}$ and which
forms an $m$--fold cover of $C$. Consider what Theorem~\ref{thm:lower}
says about the geometry of $C_m$.

By Theorem~\ref{thm:lower}, there exists an integer $n(m) \geq m$, such that 
\[
\area(\bdy C_m) \: > \:   \frac{ d_{\arc} (\psi^{n(m)})}{450\, \chi(F)^4 }
\quad \mbox{and} \quad
\height(\bdy C_m) \: > \: \frac{ d_{\arc} (\psi^{n(m)})}{536 \, \chi(F)^4 } .
\]
Because $C_m$ is an $m$--fold cover of $C$, both its area and its
height are $m$ times larger than those of $C$. Thus, for all $m \geq
1$,
\[
\area(\bdy C) 
\: > \:   \frac{ d_{\arc} (\psi^{n(m)})}{m \cdot 450\,  \chi(F)^4 } 
\: \geq \:  \frac{ d_{\arc} (\psi^{n(m)})}{n(m) \cdot 450\,  \chi(F)^4 }
\: \geq \:  \inf_{r \geq n(m)} \, \frac{ d_{\arc} (\psi^r)}{r \cdot 450\,  \chi(F)^4 }.
\]
Since every $m \geq 1$ gives rise to an integer $n(m) \geq m$
satisfying the above inequality, we have
\[
\area(\bdy C) 
\: \geq \:  \liminf_{n \to \infty} \, \frac{ d_{\arc} (\psi^n)}{n \cdot 450\,  \chi(F)^4 }
\: = \:   \frac{ \dbar_{\arc} (\psi)}{ 450\,  \chi(F)^4 }.
\]
The last equality holds because the limit in equation
\eqref{eq:stable-distance}, which defines the stable translation
distance, always exists~\cite[Section 6.6]{bridson-haefliger}.  An
identical calculation goes through for $\height(\bdy C)$.
\end{proof}


\section{Lower bound: quasi-Fuchsian manifolds}\label{sec:lower-qf}

In this section, we prove the lower bounds of Theorem
\ref{thm:main-qf}. The argument uses almost exactly the same
ingredients as the proof of Theorem~\ref{thm:lower}. The one
additional ingredient is the ability to approximate the convex core
boundary by surfaces pleated along triangulations.

\begin{prop}\label{prop:core-bdy-approx}
Let $N \cong F \times \RR$ be a cusped quasi-Fuchsian $3$--manifold,
and $p$ a puncture of $F$. As in Definition~\ref{def:distance-qf}, let
$\Delta_-(N)$ be the simplex of $\arc(F,p)$ whose vertices are the
shortest arcs from $p$ to $p$ in the lower core boundary $\bdy_-
\core(N)$. Then there is an ideal triangulation $\tau$ of $F$, such
that in the pleated surface $F_\tau$ pleated along $\tau$, each
shortest arc from $p$ to $p$ is distance at most $1$ from
$\Delta_-(N)$.

The same statement holds for the simplex $\Delta_+ (N)$.
\end{prop}

\begin{proof}
Let $R_0 \subset \bdy_- \core(N)$ be a horospherical cusp neighborhood
about puncture $p$, such that $\area(R_0) = 1$. Note that by Lemma
\ref{lemma:pleated-cusp-size}, the neighborhood $R_0$ is embedded. For
any arc $a$ from puncture $p$ to $p$ in $F$, let $\ell_0(a)$ be the
length of the geodesic representing $a$, in the hyperbolic metric of
$\bdy_- \core(N)$, outside the cusp neighborhood $R_0$. Let $\ell_N(a)$ be 
the length of the geodesic representing $a$ in the $3$--manifold $N$, outside a cusp neighborhood of longitude $1$. 
Define the set
\begin{equation}
\label{eq:s-define}
T(N) \: = \: \left\{ a \in \arc^{(0)}(F,p) \st \mbox{both ends of $a$ are at $p$, } \:
\ell_N(a) < 2 \ln \abs{6 \chi(F) }+ \ln(2) \right\}.
\end{equation}

Note that by Lemmas ~\ref{lemma:pleated-cusp} and~\ref{lemma:length-bound}, every arc $a \in \Delta_-(N)$ has length
\[
\ell_N(a) \: \leq \: \ell_0(a) \: \leq \: 2 \ln \abs{6 \chi(F) } .
\]
Thus the simplex $\Delta_-(N)$
of shortest arcs in $\bdy_- \core(N)$ must be contained in
$T(N)$. Also note that $T(N)$ must be a finite set: one way to see
this is to recall (e.g.\ from Lemma~\ref{lemma:horoball-shadows}) 
that every arc of bounded
length makes a definite contribution to $\area (\bdy C \cap \core(N))$.

Now, we apply Theorem~\ref{thm:core-bdy-teich}: there exists a
sequence of triangulations $\tau_i$ of $F$, such that the hyperbolic
metrics on the pleated surfaces $F_{\tau_i}$ converge in the
Teichm\"uller space $\teich(F)$ to $\bdy_- \core(N)$. For each $i$,
let $R_i \subset F_{\tau_i}$ be an embedded cusp neighborhood of area
$1$. For each arc $a \subset F$, let $\ell_i(a)$ be the length of the
geodesic representing $a$, in the induced hyperbolic metric on
$F_{\tau_i}$, relative to the cusp neighborhood $R_i$. Note that
$\ell_N(a) \leq \ell_i(a)$ by Lemma~\ref{lemma:pleated-cusp}.  Then,
because the metrics on $F_{\tau_i}$ converge to that on $\bdy_-
\core(N)$, the length of any arc also converges. In particular,
because $T(N)$ is a finite set of arcs, there is some $k \gg 0$ such
that
\begin{equation}\label{eq:length-change-teich}
\abs{ \ell_k(a) - \ell_0(a) } < \ln(2)/3, \qquad \forall a \in T(N).
\end{equation}

Let $\tau = \tau_k$, and let $b$ be any shortest arc on $F_\tau =
F_{\tau_k}$. By Lemma~\ref{lemma:length-bound}, $\ell_k(b) \leq 2 \ln
\abs{6 \chi(F) }.$ Furthermore, $\ell_N(b) \leq \ell_k(b)$, since the pleating 
map that produces $F_{\tau_k}$ is $1$--Lipschitz.
Then, by equation \eqref{eq:s-define}, it follows that $b \in T(N)$. Thus,
for any arc $a$ that is shortest on $\bdy_- \core(N)$ (that is, $a \in
\Delta_-(N)$), we obtain
\begin{equation}\label{eq:double-change}
\ell_0(b) - \ln(2)/3 \: < \: \ell_k(b) 
   \: \leq \: \ell_k(a) \: < \: \ell_0(a) +  \ln(2)/3,
\end{equation}
where the outer inequalities follow by \eqref{eq:length-change-teich}
and the middle inequality holds because $b$ is shortest on $F_\tau$. A
simpler way to state the conclusion of \eqref{eq:double-change} is
that on the hyperbolic surface $\bdy_- \core(N)$,
\[
\ell_0(b) \: < \: \ell_0(a) + 2 \ln(2)/3.
\]
Since $a$ is shortest on $\bdy_- \core(N)$, and $b$ is nearly
shortest, Lemma~\ref{lemma:shortish-disjoint} implies that $a$
and $b$ are disjoint (or the same arc) for any $a \in \Delta_-(N)$.
Thus $b$ is distance at most $1$ from $\Delta_-(N)$.
\end{proof}

We can now begin proving the lower bound of
Theorem~\ref{thm:main-qf}. Applying the same ideas as in
Section~\ref{sec:lower} gives the following analogue of
Lemma~\ref{lemma:arc-sequence}.

\begin{lemma}
\label{lemma:qf-arc-sequence}
Let $N \cong F \times \RR$ be a cusped quasi-Fuchsian $3$--manifold,
and $p$ a puncture of $F$. There is a sequence $a_0, a_1, \ldots, a_k$
of arcs embedded in $F$, and an associated sequence of simplicial
hyperbolic surfaces $F_{t(a_i)} \subset N$ with at most one singular
point of cone angle $2\pi \leq \theta_t < 4\pi$, such that the
following hold:
\begin{enumerate}
\item 
Each $a_i$ is short on $F_{t(a_i)}$, in the sense of Definition
\ref{def:short-arc}.
\item
\label{item:non-redundant} 
The arcs $a_0, \ldots, a_k$ are distinct up to isotopy.
\item 
$F_{t(a_0)}$ is the lower core boundary $\bdy_- \core(N)$, and
$F_{t(a_k)} = \bdy_+ \core(N)$.
\item 
Each $a_i$ is disjoint from $a_{i-1}$. In other words, $[a_{i-1},
a_i]$ is an edge of $\arc(F,p)$.
\end{enumerate}
\end{lemma}

\begin{proof}
As in Definition~\ref{def:distance-qf}, let $\Delta_-(N)$ be the
simplex of $\arc(F,p)$ whose vertices are the short arcs on the lower
boundary $\bdy_- \core(N)$. By Proposition~\ref{prop:core-bdy-approx},
there is a triangulation $\tau$ of $F$, and a pleated surface $F_\tau$
pleated along $\tau$, such that a short arc on this pleated surface
either belongs to the simplex $\Delta_-(N)$, or is at distance $1$
from some vertex of $\Delta_-(N)$.  Similarly, there is a
triangulation $\tau'$ of $F$, and a pleated surface $F_{\tau'}$
pleated along $\tau'$, whose short arc either belongs to the simplex
$\Delta_+(N)$, or is at distance $1$ from some vertex of
$\Delta_+(N)$.

By Proposition~\ref{prop:qf-sweepout}, there is a geometric sweepout
$\Psi \from F \times [0, r] \to N$, where each $F_t = \Psi(F \times \{
t\})$ is a hyperbolic cone surface with at most one singular point of
cone angle $2\pi \leq \theta_t < 4\pi$. Furthermore, $F_0 = F_\tau$
and $F_r = F_{\tau'}$, for the given triangulations $\tau$ and
$\tau'$.

Next, we apply the argument of Lemma~\ref{lemma:arc-sequence}. In the
quasi-Fuchsian setting, the proof simplifies in several ways.  There
is no need to worry about equivariance, and the set of arcs $\{a :
S(a) \neq \emptyset \}$ is finite because it is contained in $T(N)$
from equation \eqref{eq:s-define}.  We thus obtain a sequence of arcs
$a_1, \ldots, a_{k-1}$ with the following properties.
\begin{itemize}
\item 
Each $a_i$ is short on some surface $F_{t(a_i)}$ in the sweepout.
\item 
Arc $a_1$ is short on $F_0 = F_\tau$, and $a_{k-1}$ is short on $F_r =
F_{\tau'}$.
\item 
Each $a_i$ is disjoint from $a_{i-1}$. In other words, $[a_{i-1},
a_i]$ is an edge of $\arc(F,p)$.
\end{itemize}

Next, we extend this sequence of arcs to the convex core boundary. If
$a_1 \notin \Delta_-(N)$, then 
there is an arc $a_0 \in \Delta_-(N)$,
which is shortest on $\bdy_- \core(N)$ by definition, and such that
$[a_0, a_1]$ is an edge of $\arc(F,p)$. Otherwise, if $a_1 \in
\Delta_-(N)$, then we simply shift indices by $1$, so that $a_1$
becomes $a_0$. Similarly, if $a_{k-1} \notin \Delta_+(N)$, then we add
an arc $a_k \in \Delta_+(N)$, which is shortest on $\bdy_+
\core(N)$. Otherwise, if $a_{k-1} \in \Delta_+(N)$, then we simply
redefine $k := k-1$, and stop the sequence there.

We now have a sequence of arcs $a_0, \ldots, a_k$, with associated
simplicial hyperbolic surfaces $F_{t(a_i)}$, so that this sequence
satisfies all the conclusions of the lemma except possibly
\eqref{item:non-redundant}. That is, some of the $a_i$ might be in the
same isotopy class.  But if the arcs $a_0, \ldots, a_k$ are the
vertices of a path in $\arc^{(1)}(F,p)$, then some subcollection of
the $a_i$ give an embedded path.  This means that
\eqref{item:non-redundant} is satisfied, and the proof is complete.
\end{proof}

\begin{lemma}\label{lemma:qf-disks}
Let $N \cong F \times \RR$ be a cusped quasi-Fuchsian manifold, and
$p$ a puncture of $F$. Let $C \subset N$ be a horocusp corresponding
to the puncture $p$, whose longitude is $\lambda = 2^{1/4}$. Then, for
some $k \geq d_\arc(N,p)$, the annulus $A = \bdy C$ contains $2k+2$
disjoint disks, each of radius
$$r = \frac{\sqrt{2}}{8 \pi^2 \, \chi(F)^2},$$
such that the center of each disk is in the convex core $\core(N)$.
\end{lemma}

Recall that by Lemma~\ref{lemma:waist}, there is indeed an embedded
horocusp of longitude $\lambda = 2^{1/4}$. Recall as well, from
Definition~\ref{def:distance-qf}, that $d_\arc(N,p)$ is defined to be
the shortest distance in $\arc(F,p)$ between a vertex of $\Delta_-(N)$
and a vertex of $\Delta_+(N)$.

\begin{proof}[Proof of Lemma~\ref{lemma:qf-disks}]
The sequence of arcs $a_0, \ldots, a_k$, constructed in Lemma
\ref{lemma:qf-arc-sequence}, is a walk through the $1$--skeleton of
$\arc(F,p)$ from a vertex of $\Delta_-(N)$ to a vertex of
$\Delta_+(N)$. Thus, by Definition~\ref{def:distance-qf}, $k \geq
d_\arc(N,p)$.

For each $i \in \{0, \ldots, k \}$, let $g_{a_i}$ be the geodesic in
$N$ in the homotopy class of $a_i$. Then, Lemma
\ref{lemma:horoball-shadows} guarantees that there is a pair of
disjoint disks $D_{a_i}$ and $D'_{a_i}$, of radius $r = \sqrt{2} / 8
\pi^2 \, \chi(F)^2$, whose centers are the endpoints of $g_{a_i}$ on
the cusp annulus $A$.  Since the geodesic $g_{a_i}$ is contained in
the convex core of $N$, the centers of $D_{a_i}$ and $D'_{a_i}$ lie in
$\core(N)$ as well.

Finally, Lemma~\ref{lemma:horoball-shadows} also implies that if $j
\neq i$, the disks of $a_i$ are disjoint from those of $a_j$. Thus we
have at least $2k+2$ disks in total.
\end{proof}

\begin{proof}[Proof of Theorem~\ref{thm:main-qf}, lower bound]
Let $p$ be a puncture of $F$, and let $C \subset N$ be a horospherical
cusp corresponding to the puncture $p$, whose longitude is $\lambda =
2^{1/4}$. Note that by Lemma~\ref{lemma:waist}, $C$ is contained in
the maximal cusp about puncture $p$. Thus lower bounds on the area and
height of $\bdy C \cap \core(N)$ also apply to the maximal cusp.

As in Section~\ref{sec:upper} and~\ref{sec:upper-qf}, we may place
Euclidean coordinates on the annulus $A = \bdy C \cong S^1 \times
\RR$, in which the $\RR$ direction is vertical. Let $B \subset A$ be a
compact annular band, with boundary consisting of horizontal circles,
which is the smallest such band that contains all of the $2k+2$ disks
of Lemma~\ref{lemma:qf-disks}. Since each disk has radius $r =
\sqrt{2} / 8 \pi^2 \, \chi(F)^2$, B\"or\"oczky's estimate on the
density of a circle packing~\cite[Theorem 1]{boroczky} implies that
$$\area(B) \: \geq \: 
\frac{2 \sqrt{3}}{\pi} \cdot (2d_{\arc} (N) + 2) \cdot \pi \left(  \frac{\sqrt{2}}{ 8 \pi^2 \, \chi(F)^2} \right)^{\!\! 2}
\: = \:  \frac{ \sqrt{3} \, (d_{\arc} (N) + 1) }{8 \pi^4 \, \chi(F)^4} 
\: > \: \frac{ d_{\arc} (N) + 1 }{450 \, \chi(F)^4} \, .
$$ 
Similarly, since $\area(B) = \lambda \cdot \height(B)$, and we have
normalized the horocusp so that $\lambda = 2^{1/4}$, we have
$$\height(B)
\: \geq \:  \frac{ \sqrt{3} \, (d_{\arc} (N) + 1) }{2^{1/4} \, 8 \pi^4 \, \chi(F)^4} 
\: > \: \frac{ (d_{\arc} (N) + 1) }{536 \, \chi(F)^4} \, .
$$

To complete the proof, it remains is to bound the difference in area
(or height) between $B$ and $\bdy C \cap \core(N)$. Note that by Lemma
\ref{lemma:qf-disks}, each of the $(2k+2)$ disks has its center inside
$\core(N)$. Therefore, the upper boundary of $B$ is at most $r =
\sqrt{2} / 8 \pi^2 \, \chi(F)^2$ higher than $\bdy_+ \core(N)$, and
the lower boundary of $B$ is at most $r$ lower than $\bdy_-
\core(N)$. This implies that
\[
\area(\bdy C \cap \core(N)) \: \geq \:  \area(B) - 2 \lambda r
\: > \: \frac{ d_{\arc} (N) }{450 \, \chi(F)^4} -  \frac{2 \cdot 2^{3/4} }{8 \pi^2 \, \chi(F)^2}
\: > \: \frac{ d_{\arc} (N)  }{450 \, \chi(F)^4} - \frac{1}{23 \, \chi(F)^2}.
\]
Similarly, 
\[
\height(\bdy C \cap \core(N)) \: \geq \:  \height(B) - 2 r
\: > \: \frac{ d_{\arc} (N)  }{536 \, \chi(F)^4} - \frac{1}{27 \, \chi(F)^2}.
\qedhere
\]
\end{proof}


\section{Covers and the arc complex}\label{sec:covers}

In this section, we will apply Theorem~\ref{thm:main-qf} to prove Theorem~\ref{thm:lifting}, which relates the arc complex of a surface $S$ to that of its cover $\Sigma$. The proof uses some classical results in Kleinian groups to construct a quasi-Fuchsian manifold with prescribed short arcs on its convex core boundary (see Lemma~\ref{lemma:prescribed-qf}). We begin by recalling some terminology and results, while pointing the reader to Marden~\cite[Chapter 3]{marden:book} for a more detailed reference.

Let $\Gamma$ be a Kleinian group with limit set $\Lambda$. The \emph{domain of discontinuity} is $\Omega = \bdy \HH^3 \setminus \Lambda$. When $N = \HH^3 / \Gamma$ is quasi-Fuchsian, $\Omega$ is the disjoint union of two open disks $\Omega_+$ and $\Omega_-$, each of which admits a conformal, properly discontinuous action by $\Gamma$. The quotients $S_\pm = \Omega_\pm / \Gamma$ are Riemann surfaces, called  the (top and bottom)  \emph{conformal boundary} of $N = \HH^3 / \Gamma$.


For each quasi-Fuchsian manifold $N$, there is a natural ``nearest point retraction'' map $r \from S_\pm \to \bdy_\pm \core(N)$. Sullivan proved that, if $S_\pm$ is given the unique hyperbolic metric in its conformal class, the map $r \from S_\pm \to \bdy_\pm \core(N)$ is $K$--Lipschitz, for a universal constant $K > 1$. Much more recently, Epstein, Marden, and Markovic~\cite{emm:quasiconformal} showed that the optimal Lipschitz constant is $2$. We will use their result for concreteness, while emphasizing that Sullivan's original $K$--Lipschitz statement is all that is truly needed.

We also recall some facts from the geometry of surfaces. Let $\eps_2$ be the $2$--dimensional Margulis constant. For any simple closed geodesic $\gamma$ in a hyperbolic surface $S$, of length $\ell = \ell(\gamma) < \eps_2$, the $\eps_2$--thin region of $S$ containing $\gamma$ is an embedded collar of radius $r(\ell)$. The function $r(\ell)$ is monotonically decreasing, and $r(\ell) \to \infty$ as $\ell \to 0$. See Buser~\cite{buser:geometry-spectra}  for explicit estimates on $\eps_2$ and $r(\ell)$.

\begin{define}\label{def:suff-thin}
Let $S$ be a hyperbolic surface, and $\gamma$ a simple closed geodesic on $S$. We say that $\gamma$ is \emph{sufficiently thin} if its length $\ell = \ell(\gamma)$ is short enough that the $\eps_2$--thin collar about $\gamma$ has radius
\begin{equation}\label{eq:thin}
r(\ell) \: > \: \ln \abs{ 6 \chi(S) / \eps_2 }.
\end{equation}
This requirement on collar radius is motivated by Lemma~\ref{lemma:length-bound}.
\end{define}


Given this background, we can prove the following constructive lemma.

\begin{lemma}\label{lemma:prescribed-qf}
Let $F$ be a surface with a preferred puncture $p$, and let $a_-, a_+ \in \arc^{(0)}(F,p)$ be arcs from $p$ to $p$. Then there exists a quasi-Fuchsian manifold $N \cong F \times \RR$, such that  $a_\pm$ is the unique shortest arc from $p$ to $p$ on $\bdy_\pm \core(N)$.
%
\end{lemma}

\begin{proof}
Let $R_0 \subset F$ be a neighborhood of the puncture $p$, and let $R(a_+)$ be a regular neighborhood of $R_0 \cup a_+$. This is topologically a pair of pants, whose frontier in $F$ consists of a pair of essential closed curves $c_+ ,c'_+$. (If $F$ is a once-punctured torus, then $c_+$ is isotopic to $c'_+$; this will not affect our arguments.)  Similarly, let $c_+ ,c'_+$ be closed curves that form the frontier of a regular neighborhood $R(a_-)$.

Choose hyperbolic metrics $X_\pm$ on $F$, in which the geodesic representatives  of $c_\pm$ and $c'_\pm$ have less than half the length required to be sufficiently thin. By Bers simultaneous uniformization ~\cite[page 136]{marden:book}, there is a quasi-Fuchsian manifold $N \cong F \times \RR$ whose top conformal boundary is $X_+$ and whose bottom conformal boundary is $X_-$. Thus, by Epstein, Marden, and Markovic~\cite{emm:quasiconformal}, the geodesic representatives of $c_\pm$ and $c'_\pm$ are sufficiently thin on $\bdy_\pm \core(N)$.

We claim that $a_+$ is the unique shortest arc from $p$ to $p$ on $\bdy_+ \core(N)$. For concreteness, we will measure lengths relative to the horospherical cusp neighborhood $Q(p) \subset \bdy_+ \core(N)$ whose boundary has length $\eps_2$. Since horoballs are convex, it follows that $Q(p)$ is $\eps_2$--thin. Then, by Lemma~\ref{lemma:length-bound}, 
there must be a geodesic $\alpha_+$ from $p$ to $p$ whose length relative to this cusp neighborhood satisfies 
$$\ell(\alpha_+) \: \leq \: 2 \ln \abs{ 6 \chi(S) / \eps_2 }.$$
Now, let $b_+$ be any arc from $p$ to $p$, other than $a_+$. Since $b_+$ is not isotopic into the neighborhood $R(a_+)$, it must cross $c_+ \cup c'_+$. 
By the Margulis lemma, the $\eps_2$--thin collars about those curves
are disjoint from the $\eps_2$--thin cusp neighborhood $Q(p)$. Thus,
since $c_+$ and $c'_+$ are sufficiently thin, and the width of a
collar is twice the radius, equation \eqref{eq:thin} implies that
$\ell(b_+) > \ell(\alpha_+)$. Therefore $a_+ = \alpha_+$, the unique
shortest arc from $p$ to $p$.

By the same argument, $a_-$ is the unique short arc on $\bdy_- \core(N)$.
\end{proof}

We can now complete the proof of Theorem~\ref{thm:lifting}, which we restate.

\begin{named}{Theorem~\ref{thm:lifting}}
Let $\Sigma$ and $S$ be surfaces with one puncture, and $f \from
\Sigma \to S$ a covering map of degree $n$.  Let $\pi \from \arc(S)
\to \arc(\Sigma)$ be the lifting relation induced by $f$. Then, for
all $a,b \in \arc^{(0)}(S)$, we have
\[
\frac{d(a,b)}{4050 \, n \, \chi(S)^6} \,  - 2 \: < \: d(\alpha, \beta) \: \leq \: d(a, b)
\]
where $\alpha \in \pi(a)$ and $\beta \in \pi(b)$.
\end{named}

\begin{proof} 
Recall that by Definition~\ref{def:lifting-map}, $\pi(a)$ is the set
of $n$ vertices in $\arc(\Sigma)$ representing arcs that project to
$a$. These vertices form a simplex in $\arc(\Sigma)$, since the arcs
that comprise $f^{-1}(a)$ are disjoint.  Similarly, if $a$ and $a'$
are distance $1$ in $\arc(S)$, then all the $2n$ lifts of $a$ and $a'$
are disjoint in $\Sigma$, hence every vertex $\alpha \in \pi(a)$ is
distance $1$ in $\arc(\Sigma)$ from every vertex $\alpha' \in
\pi(a')$.  Thus, by induction on distance, we have 
\[
d(\alpha, \beta) \, \leq \, d(a, b)
\]
for any $\alpha \in \pi(a)$ and any $\beta \in \pi(b)$.

To prove the other inequality in Theorem~\ref{thm:lifting}, assume the
cover is non-trivial: that is, $n > 1$.  Fix $a, b \in
\arc^{(0)}(\Sigma)$.  By Lemma~\ref{lemma:prescribed-qf}, there is a
quasi-Fuchsian manifold $M \cong S \times \RR$, such that $a$ is the
unique shortest arc on $\bdy_- \core(M)$ and $b$ is the unique
shortest arc on $\bdy_+ \core(M)$.

Since $S$ has a unique puncture $p$, we have $\arc(S) = \arc(S,p)$.
Let $C \subset M$ be the maximal horospherical cusp corresponding to
this unique puncture. By Definition~\ref{def:distance-qf},
$d_\arc(M,p) = d(a,b)$.  Thus, by the lower bound of
Theorem~\ref{thm:main-qf},
\begin{equation}\label{eq:lower-m} 
\frac{d(a,b)}{450 \, \chi(S)^4} - \frac{1}{23 \, \chi(S)^2} 
   \: < \: \area( \bdy C \cap \core(M)).
\end{equation}

Use the $n$--fold covering map $f \from \Sigma \to S$ to lift the
hyperbolic metric on $M \cong S \times \RR$ to a quasi-Fuchsian
structure on $N \cong \Sigma \times \RR$.  The convex core of $N$
covers the convex core of $M$.  The horocusp $C \subset M$ lifts to a
horocusp $D \subset N$, which is an $n$--sheeted cover of $C$. Thus
\begin{equation}\label{eq:m-n}
n \cdot \area (\bdy C \cap \core(M)) \: = \:  \area(\bdy D \cap \core(N)).
\end{equation}

Observe that every arc $a \in \bdy_+ \core(M)$ lifts to $n$ disjoint
arcs in $\bdy_+ \core(N)$, each of which has the same length as $a$
(outside the horocusps $C$ and $D$, respectively).  Thus the $n$ arcs
of $\pi(b)$ are shortest on $\bdy_+ \core(N)$, and similarly the $n$
arcs of $\pi(a)$ are shortest on $\bdy_- \core(N)$.  By
Definition~\ref{def:distance-qf}, this implies $d_\arc(N,p) \leq
d(\alpha, \beta)$ for any $\alpha \in \pi(a)$ and any $\beta \in
\pi(b)$.  Therefore, the upper bound of Theorem~\ref{thm:main-qf}
implies
\begin{equation}\label{eq:upper-n}
\area( \bdy D \cap \core(N)) \: < \:  9 \, \chi(\Sigma)^2 \, d(\alpha,\beta) + 
      \Big| 12 \chi(\Sigma)   \ln \abs{ \chi(\Sigma) } + 26 \chi(\Sigma) \Big|.
\end{equation}

Combining equations \eqref{eq:lower-m}, \eqref{eq:m-n} and
\eqref{eq:upper-n}, we obtain
\begin{equation*}
 \frac{n \cdot d(a, b) }{450 \,  \chi(S)^4} - \frac{n}{23  \,  \chi(S)^2} \: < \:  9 \, \chi(\Sigma)^2 \, d(\alpha,\beta) + 
      \Big| 12 \chi(\Sigma)   \ln \abs{ \chi(\Sigma) } + 26 \chi(\Sigma) \Big|,
\end{equation*}
which can be rearranged, using $\chi(\Sigma) = n \chi(S)$, to give
\begin{equation*}
\frac{d(a, b)  }{4050 \, n \, \chi(S)^6} - \frac{1}{23 \cdot 9 \, n \,  \chi(S)^4} -  \frac{\Big| 12 \chi(\Sigma)   \ln \abs{ \chi(\Sigma) } + 26 \chi(\Sigma) \Big|  } {9 \chi(\Sigma)^2} \: < \: d(\alpha, \beta).
\end{equation*}
Since $\Sigma$ is a once-punctured surface that non-trivially covers
$S$, we have $\abs{\chi(\Sigma)} \geq 3$. It follows that the additive
error on the left-hand side is bounded above by 2.  This completes the
proof. 
\end{proof}


\appendix

\section{Approximating the convex core boundary}

The goal of this appendix is to write down a proof of the following
result, which is needed in the argument of Section~\ref{sec:lower-qf}.

\begin{theorem}\label{thm:core-bdy-teich}
Let $N \cong F \times \RR$ be a cusped quasi-Fuchsian
$3$--manifold. Then there is a sequence $\tau_i$ of ideal
triangulations of $F$, such that the induced hyperbolic metrics on the
pleated surfaces $F_{\tau_i}$ converge in the Teichm\"uller space
$\teich(F)$ to the hyperbolic metric on the lower core boundary
$\bdy_- \core(N)$. Furthermore, the pleating maps for the $F_{\tau_i}$
converge in the compact--open topology to a pleating map for $\bdy_-
\core(N)$.

The same statement holds for the upper core boundary $\bdy_+ \core(N)$.
\end{theorem}


The statement of Theorem~\ref{thm:core-bdy-teich} is entirely
unsurprising, and morally it should fit into the toolbox of well-known
results about laminations and pleated surfaces~\cite[Chapters 4 and
5]{ceg:notes-on-notes}. Indeed, the standard toolbox of Kleinian group
theory leads to a relatively quick proof of the theorem.  However,
this short proof is also somewhat technical, as it requires passing
between several different topologies on spaces of laminations and
pleated surfaces.


The main reference for the following argument is Canary, Epstein, and
Green~\cite{ceg:notes-on-notes}.  See also Thurston
\cite{thurston:notes, thurston:geometry-dynamics}, Bonahon
\cite{bonahon:lamination}, and Ohshika~\cite{Ohshika05}.

\begin{define}\label{def:gl}
Let $S$ be a punctured hyperbolic surface of finite area. Let $\GL(S)$
denote the set of geodesic laminations on $S$: that is, laminations
where each leaf is a geodesic. We equip $\GL(S)$ with the
\emph{Chabauty topology}.  In this topology, a sequence $\{L_i \} $
converges to $L \in \GL(S)$ if and only if:
\begin{enumerate}
\item 
If a subsequence $x_{n_i} \in L_{n_i}$ converges to $x \in S$, then $x
\in L$.
\item 
For all $x \in L$, there exists a sequence $x_i \in L_i$, such that
$x_i \to x$.
\end{enumerate} 
The Chabauty topology is metrizable
\cite[Proposition~3.1.2]{ceg:notes-on-notes}. In fact, when restricted
to compact laminations, the Chabauty topology reduces to be the
Hausdorff topology (induced by the Hausdorff distance between compact
sets).  See~\cite[Section 3.1]{ceg:notes-on-notes} for more details.
\end{define}

Definition~\ref{def:gl} makes use of a hyperbolic metric on $S$, but
in an inessential way. If we modify a metric $d$ to a new hyperbolic
metric $d'$, each lamination $L \in \GL(S)$ that is geodesic in $d$
can be straightened to a geodesic lamination of $d'$. This
straightening does not affect convergence of laminations. Thus the
space $\GL(S)$ only depends on the topology of $S$.

We use the term \emph{curve} to denote a simple closed geodesic in
$S$. The following lemma is a good example of convergence in the
Chabauty topology.

\begin{lemma}
\label{Lem:Twist}
For every curve $\alpha \subset S$, there is a sequence of ideal
triangulations $\tau_i$, converging in the Chabauty topology to a
lamination $\alpha' \supset \alpha$.
\end{lemma}

\begin{proof}
Fix any ideal triangulation $\tau$.  Let $D = D_\alpha$ be a Dehn
twist about $\alpha$.  Then $\tau_i = D^i(\tau)$ converges to the desired
$\alpha'$.
\end{proof}


\begin{define}
\label{def:ml}
Let $S$ be a punctured hyperbolic surface of finite area. Then
$\ML(S)$ denotes the space of compact, transversely measured
laminations. Every point of $\ML(S)$ is a pair $(\lamm, \mu)$ where
$\lamm$ is a compact geodesic lamination and $\mu$ is a transverse
measure of full support. That is, for each arc $\alpha$ intersecting
$\lamm$ transversely, $\mu(\alpha)$ is a positive number that stays
invariant under an isotopy preserving the leaves of $\lamm$. The
natural topology on $\ML(S)$ is called the \emph{measure topology}.

Let $\PML(S)$ denote the projectivization of $\ML(S)$, in which
nonzero measures that differ by scaling become identified. The measure
topology on $\ML(S)$ descends to $\PML(S)$.
\end{define}

If $\lamm$ is a disjoint union of arcs and closed curves in $S$, an
example of a transverse measure is the \emph{counting measure}, where
$\mu(\alpha) = \abs{\alpha \cap \lamm}$. For another example, suppose
that $S = \bdy_+ \core(N)$ is the upper boundary of the convex core in
a quasi-Fuchsian $3$--manifold. Then the pleating lamination $\lamm$
has a \emph{bending measure}, where $\mu(\alpha)$ is the integral of
the bending of $\alpha$ as it crosses leaves of $\lamm$.

The measure-forgetting map $\PML(S) \to \GL(S)$ is not continuous, but
it has the following convenient property.

\begin{fact}
\label{Fac:OverConverge}
Suppose that $(\lamm_i, \mu_i) \to (\lamm, \mu) \in \PML(S)$, in the
measure topology.  Then, after passing to a subsequence, there is a
lamination $\lamm' \in \GL(S)$ so that $\lamm \subset \lamm'$ and
$\lamm_i \to \lamm'$ in the Chabauty topology.
\end{fact}



With these facts in hand, we can prove
Theorem~\ref{thm:core-bdy-teich}. The proof contains two steps, the
first of which deals with laminations only.

\begin{lemma}
\label{lemma:OverConverge}
Let $\lamm \in \GL(S)$ be a measurable lamination. Then there is a
sequence of ideal triangulations $\tau_i$ converging in the Chabauty
topology to a lamination $\lamm'' \supset \lamm$.
\end{lemma}

\begin{proof}
Pick $\mu$, a measure of full support on $\lamm$, so that $(\lamm,
\mu) \in \PML(S)$.

Thurston proved that
curves, equipped with the counting measure, are dense in $\PML(S)$
\cite[Proposition 15]{bonahon:lamination}.  Thus we may pick a
sequence of curves $\{ \alpha_i \}$ converging to $(\lamm, \mu)$ in
the measure topology.  By Fact~\ref{Fac:OverConverge}, we may pass to a
subsequence and reindex so that $\alpha_i$ converges, in the Chabauty
topology, to a lamination $\lamm'$ containing $\lamm$.

By Lemma~\ref{Lem:Twist}, we may choose $\{\tau_{i,j}\}$, a sequence of
sequences of ideal triangulations, so that for all $i$
\[
\tau_{i,j} \to \alpha'_i \supset \alpha_i
\quad \mbox{as} \quad j \to \infty.
\]
For all $i$, choose an increasing function $j \from \NN \to \NN$, so
that $\tau_{i,j(i)}$ and $\alpha'_i$ have distance at most $1/i$, in
the metric that induces the Chabauty topology.

\begin{proofclaim}
The sequence $\tau_{i,j(i)}$ contains a subsequence converging to
$\lamm'' \supset \lamm'$.
\end{proofclaim}

\begin{proof}[Proof of claim]
As $\GL(S)$ is compact~\cite[Proposition~4.1.6]{ceg:notes-on-notes},
the sequence $\tau_{i,j(i)}$ contains a convergent subsequence.  In an
abuse of notation, denote this convergent subsequence by $\tau_i =
\tau_{i,j(i)}$.  Let $\lamm''$ be the limit of the $\tau_i$.  Note
that $\lamm''$ is not compact.

Recall that the Chabauty distance between $\tau_i$ and $\alpha'_i$ is
at most $1/i$.  Thus the sequence $\alpha'_i$ has the same limit as
$\tau_i$, namely $\lamm''$.  Since $\alpha_i \subset \alpha'_i$, it
follows by~\cite[Lemma 4.1.8]{ceg:notes-on-notes} that the limit of
$\alpha_i$ must be contained in the limit of $\alpha'_i$.  That is,
$\lamm' \subset \lamm''$, as desired.
\end{proof}

Since $\lamm \subset \lamm'$, Lemma~\ref{lemma:OverConverge} is
proven.
\end{proof}

The second step of the argument connects the above discussion of
laminations to pleated surfaces.  Following Definition
\ref{def:pleated}, we call a lamination $L \in \GL(S)$
\emph{realizable} if it is realized by a pleating map $f \from S \to
N$ homotopic to a prescribed map $f_0$.

\begin{lemma}
\label{lemma:teich-converge}
Let $N \cong S \times \RR$ be a cusped quasi-Fuchsian $3$--manifold,
and fix an embedding $f_0 \from S \to S \times \{ 0 \}$.  Let $L_i \in
\GL(S)$ be a sequence of laminations on $S$, which are realizable by
pleating maps homotopic to $f_0$. Then, if $L_i \to L'$ in the
Chabauty topology, and $L \subset L'$ is also realizable, the
hyperbolic metrics $d_i$ induced by pleating along $L_i$ converge in
$\teich(S)$ to the hyperbolic metric $d$ induced by pleating along
$L$.

In fact, the pleating maps $f_i \from S \to N$ converge to the
pleating map $f \from S \to N$ that realizes $L$, in the space
$\mathcal{MPS}(S, N)$ of marked pleated surfaces homotopic to $f_0$.
\end{lemma}

We refer the reader to~\cite[Definition 5.2.14]{ceg:notes-on-notes}
for the full definition of $\mathcal{MPS}(S, N)$. It suffices to note
that convergence in $\mathcal{MPS}(S, N)$ involves \emph{both}
convergence of metrics in $\teich(S)$ and convergence of pleating maps
in the compact-open topology.  

\begin{proof}[Proof of Lemma~\ref{lemma:teich-converge}]
Let $C$ be an embedded neighborhood of the cusps of $N$. Then $K
:= \core(N) \setminus C$ is a compact set, and every pleated
surface homotopic to $f_0$ must intersect $K$.

With this notation,~\cite[Theorem 5.2.18]{ceg:notes-on-notes} implies
that the space $\mathcal{MPS}(S, N) = \mathcal{MPS}(S, K)$ of marked
pleated surfaces that intersect $K$ is compact. In particular, the
pleating maps $f_i \from S \to N$, which pleat along lamination $L_i$,
have a convergent subsequence, $f_{n_i} \to f_\infty$. Let $L_\infty$
be the pleating lamination of $f_\infty$.

Now, recall that $L_i \to L'$ in the Chabauty topology, and $L \subset
L'$ is realizable by a pleating map $f \from S \to N$ homotopic to
$f_0$. Since every component of $f(S \setminus L)$ is totally geodesic
in $N$, the leaves of $L' \setminus L$ are mapped into these totally
geodesic regions, hence $L'$ is realized by the same map $f$. In this
setting, Ohshika~\cite[Lemma 1.3]{Ohshika05} notes that $f_\infty = f$
is the same pleating map that realizes the limiting lamination $L'$.

Therefore, every convergent subsequence of $f_i$ must limit to the
same map $f_\infty = f$ that realizes the lamination $L$. Since
$\mathcal{MPS}(S, K)$ is compact, this means that $f_i \to f$ in the
topology on $\mathcal{MPS}(S, K)$. This means that $f_i \to f$ in the
compact-open topology, and also that the hyperbolic metrics $d_i$
induced by $f_i$ converge in $\teich(S)$ to the hyperbolic metric $d$
induced by $f$.\end{proof}

\begin{proof}[Proof of Theorem~\ref{thm:core-bdy-teich}]
Let $N \cong F \times \RR$ be a cusped quasi-Fuchsian $3$--manifold,
as in the statement of the theorem. Let $L$ be the pleating lamination
on the lower core boundary $\bdy_- \core(N)$. The bending measure
$\mu$ on $\bdy_- \core(N)$ is a transverse measure of full support, so
$L$ is a measurable lamination. By Lemma~\ref{lemma:OverConverge},
there is a sequence of ideal triangulations $\tau_i \to L'$ in the
Chabauty topology, where $L \subset L'$. Now, by Lemma
\ref{lemma:teich-converge}, the pleating maps $f_i \from F \to N$ that
pleat along $\tau_i$ induce hyperbolic metrics on $F$ that converge in
$\teich(F)$ to the induced metric on $\bdy_- \core(N)$. By the definition of 
$\mathcal{MPS}(S, N)$, these pleating maps also converge in the 
compact--open topology to a pleating map for $\bdy_- \core(N)$.
\end{proof}


\section{A lemma in point-set topology}

Let $X$ be a topological space, and let $\mathcal{S}$ be a cover of
$X$ by closed sets. Define a \emph{discrete walk of length $k$}
through sets of $\mathcal{S}$ to be a sequence of points $x_0, x_1,
\ldots, x_k$, such that $x_0, x_1 \in S_1$, $x_1, x_2 \in S_2$, and so
on, for sets $S_1, \ldots, S_k$ that belong to $\mathcal{S}$. We say
this sequence is a walk from $x_0$ to $x_k$.

The following observation is needed in the proof of
Lemma~\ref{lemma:arc-sequence}.

\begin{lemma}
\label{lemma:appendix}
Let $X$ be a connected topological space, and let $S_1, \ldots, S_n$
be closed sets whose union is $X$. Then, for any pair of points $x,y
\in X$, there is a discrete walk from $x$ to $y$ through sets in the
collection $\{ S_1, \ldots, S_n \}$.
\end{lemma}

\begin{proof}
Define an equivalence relation $\equiv$ on $X$, where
\[
x \equiv y \quad \Leftrightarrow \quad \mbox{there exists a discrete
walk from $x$ to $y$ through $\{ S_1, \ldots, S_n \}$.}
\]
Reversing and concatenating walks proves this is an equivalence
relation.  Furthermore, each closed set $S_i$ must be entirely
contained in an equivalence class, because its points are connected by
a walk of length $1$. Thus each equivalence class is closed. Since the
connected space $X$ cannot be expressed as a disjoint union of
finitely many closed sets, all of $X$ must be in the same equivalence
class.
\end{proof}

\bibliographystyle{hamsplain} 
\bibliography{biblio.bib}

\providecommand{\bysame}{\leavevmode\hbox to3em{\hrulefill}\thinspace}
\providecommand{\href}[2]{#2}
\begin{thebibliography}{10}

\bibitem{adams:waist2}
Colin~C. Adams, \emph{Waist size for cusps in hyperbolic 3-manifolds {II}},
  Preprint.

\bibitem{agol:virtual-haken}
Ian Agol, \emph{{The virtual Haken conjecture}}, \mbox{arXiv:1204.2810}, With
  an appendix by Ian Agol, Daniel Groves, and Jason Manning.

\bibitem{agol:6theorem}
\bysame, \emph{Bounds on exceptional {D}ehn filling}, Geom. Topol. \textbf{4}
  (2000), 431--449 (electronic).

\bibitem{agol:small}
\bysame, \emph{Small 3-manifolds of large genus}, Geom. Dedicata \textbf{102}
  (2003), 53--64.

\bibitem{akiyoshi-miyachi-sakuma}
Hirotaka Akiyoshi, Hideki Miyachi, and Makoto Sakuma, \emph{Variations of
  {M}c{S}hane's identity for punctured surface groups}, Spaces of {K}leinian
  groups, London Math. Soc. Lecture Note Ser., vol. 329, Cambridge Univ. Press,
  Cambridge, 2006, pp.~151--185.

\bibitem{bonahon:lamination}
Francis Bonahon, \emph{Geodesic laminations on surfaces}, Laminations and
  foliations in dynamics, geometry and topology ({S}tony {B}rook, {NY}, 1998),
  Contemp. Math., vol. 269, Amer. Math. Soc., Providence, RI, 2001, pp.~1--37.

\bibitem{boroczky}
K\'{a}roly B{\"o}r{\"o}czky, \emph{Packing of spheres in spaces of constant
  curvature}, Acta Math. Acad. Sci. Hungar. \textbf{32} (1978), no.~3-4,
  243--261.

\bibitem{bridson-haefliger}
Martin~R. Bridson and Andr{\'e} Haefliger, \emph{Metric spaces of non-positive
  curvature}, Grundlehren der Mathematischen Wissenschaften [Fundamental
  Principles of Mathematical Sciences], vol. 319, Springer-Verlag, Berlin,
  1999.

\bibitem{brock:fibered}
Jeffrey~F. Brock, \emph{Weil--{P}etersson translation distance and volumes of
  mapping tori}, Comm. Anal. Geom. \textbf{11} (2003), no.~5, 987--999.

\bibitem{brock:quasifuchsian}
\bysame, \emph{The {W}eil-{P}etersson metric and volumes of 3-dimensional
  hyperbolic convex cores}, J. Amer. Math. Soc. \textbf{16} (2003), no.~3,
  495--535 (electronic).

\bibitem{brock-canary-minsky:elc}
Jeffrey~F. Brock, Richard~D. Canary, and Yair~N. Minsky, \emph{{The
  classification of Kleinian surface groups, II: The Ending Lamination
  Conjecture}}, 2004, \mbox{arXiv:math/0412006}.

\bibitem{buser:geometry-spectra}
Peter Buser, \emph{Geometry and spectra of compact {R}iemann surfaces},
  Progress in Mathematics, vol. 106, Birkh\"auser Boston Inc., Boston, MA,
  1992.

\bibitem{canary:covering}
Richard~D. Canary, \emph{A covering theorem for hyperbolic 3--manifolds and its
  applications}, Topology \textbf{35} (1996), no.~3, 751--778.

\bibitem{ceg:notes-on-notes}
Richard~D. Canary, David B.~A. Epstein, and Paul Green, \emph{Notes on notes of
  {T}hurston}, Analytical and geometric aspects of hyperbolic space
  ({C}oventry/{D}urham, 1984), London Math. Soc. Lecture Note Ser., vol. 111,
  Cambridge Univ. Press, Cambridge, 1987, pp.~3--92.

\bibitem{cfp:tunnels}
Daryl Cooper, David Futer, and Jessica~S. Purcell, \emph{Dehn filling and the
  geometry of unknotting tunnels}, Geom. Topol. (to appear),
  \mbox{arXiv:1105.3461}.

\bibitem{chk:orbifold}
Daryl Cooper, Craig~D. Hodgson, and Steven~P. Kerckhoff,
  \emph{Three-dimensional orbifolds and cone-manifolds}, MSJ Memoirs, vol.~5,
  Mathematical Society of Japan, Tokyo, 2000, With a postface by Sadayoshi
  Kojima.

\bibitem{snappea}
Marc Culler, Nathan~M. Dunfield, and Jeffrey~R. Weeks, \emph{{SnapPy, a
  computer program for studying the geometry and topology of 3-manifolds}},
  {\tt http://\allowbreak snappy.\allowbreak computop.\allowbreak org}.

\bibitem{emm:quasiconformal}
David B.~A. Epstein, Albert Marden, and Vladimir Markovic, \emph{Quasiconformal
  homeomorphisms and the convex hull boundary}, Ann. of Math. (2) \textbf{159}
  (2004), no.~1, 305--336.

\bibitem{floyd-hatcher}
William Floyd and Allen Hatcher, \emph{Incompressible surfaces in
  punctured-torus bundles}, Topology Appl. \textbf{13} (1982), no.~3, 263--282.

\bibitem{fkp:volume}
David Futer, Efstratia Kalfagianni, and Jessica~S. Purcell, \emph{{Dehn
  filling, volume, and the Jones polynomial}}, J. Differential Geom.
  \textbf{78} (2008), no.~3, 429--464.

\bibitem{fkp:farey}
\bysame, \emph{{Cusp areas of {F}arey manifolds and applications to knot
  theory}}, Int. Math. Res. Not. IMRN \textbf{2010} (2010), no.~23, 4434--4497.

\bibitem{gabai:tameness-notes}
David Gabai, \emph{Hyperbolic geometry and 3-manifold topology}, Low
  dimensional topology, IAS/Park City Math. Ser., vol.~15, Amer. Math. Soc.,
  Providence, RI, 2009, pp.~73--103.

\bibitem{gf:punctured-torus}
Fran\c{c}ois Gu\'eritaud and David Futer~(appendix), \emph{On canonical
  triangulations of once-punctured torus bundles and two-bridge link
  complements}, Geom. Topol. \textbf{10} (2006), 1239--1284.

\bibitem{hatcher:triangulations}
Allen Hatcher, \emph{On triangulations of surfaces}, Topology Appl. \textbf{40}
  (1991), no.~2, 189--194.

\bibitem{hatcher-thurston:2bridge}
Allen Hatcher and William~P. Thurston, \emph{Incompressible surfaces in
  {$2$}-bridge knot complements}, Invent. Math. \textbf{79} (1985), no.~2,
  225--246.

\bibitem{irmak-mccarthy:arc-complex}
Elmas Irmak and John~D. McCarthy, \emph{Injective simplicial maps of the arc
  complex}, Turkish J. Math. \textbf{34} (2010), no.~3, 339--354.

\bibitem{lackenby:surgery}
Marc Lackenby, \emph{Word hyperbolic {D}ehn surgery}, Invent. Math.
  \textbf{140} (2000), no.~2, 243--282.

\bibitem{marden:book}
Albert Marden, \emph{Outer circles: An introduction to hyperbolic 3-manifolds},
  Cambridge University Press, Cambridge, 2007.

\bibitem{minsky:models-bounds}
Yair~N. Minsky, \emph{{The classification of Kleinian surface groups, I: Models
  and bounds}}, Ann. of Math. (2) \textbf{171} (2010), 1--107.

\bibitem{mostow:rigidity}
George~D. Mostow, \emph{Strong rigidity of locally symmetric spaces}, Princeton
  University Press, Princeton, N.J., 1973, Annals of Mathematics Studies, No.
  78.

\bibitem{Ohshika05}
Ken'ichi Ohshika, \emph{Kleinian groups which are limits of geometrically
  finite groups}, Mem. Amer. Math. Soc. \textbf{177} (2005), no.~834, xii+116.

\bibitem{otal:fibered}
Jean-Pierre Otal, \emph{Le th\'eor\`eme d'hyperbolisation pour les vari\'et\'es
  fibr\'ees de dimension 3}, Ast\'erisque (1996), no.~235, x+159.

\bibitem{prasad:rigidity}
Gopal Prasad, \emph{Strong rigidity of {${\bf Q}$}-rank {$1$} lattices},
  Invent. Math. \textbf{21} (1973), 255--286.

\bibitem{przytycki-wise:mixed-manifolds}
Piotr Przytycki and Daniel~T. Wise, \emph{Mixed 3--manifolds are virtually
  special}, 2012, \mbox{arXiv:1205.6742}.

\bibitem{rafi-schleimer:covers}
Kasra Rafi and Saul Schleimer, \emph{Covers and the curve complex}, Geom.
  Topol. \textbf{13} (2009), no.~4, 2141--2162.

\bibitem{thurston:notes}
William~P. Thurston, \emph{The geometry and topology of three-manifolds},
  Princeton Univ. Math. Dept. Notes, 1980, {\tt http://\allowbreak
  www.msri.org/\allowbreak gt3m/}.

\bibitem{thurston:survey}
\bysame, \emph{Three-dimensional manifolds, {K}leinian groups and hyperbolic
  geometry}, Bull. Amer. Math. Soc. (N.S.) \textbf{6} (1982), no.~3, 357--381.

\bibitem{thurston:geometry-dynamics}
\bysame, \emph{On the geometry and dynamics of diffeomorphisms of surfaces},
  Bull. Amer. Math. Soc. (N.S.) \textbf{19} (1988), no.~2, 417--431.

\end{thebibliography}

\end{document}